\documentclass[11pt]{amsart}
\usepackage{amsmath,amsfonts,amssymb,enumerate,mathrsfs,mathtools,amscd}
\usepackage{pgfplots}

\pgfplotsset{width=6cm,compat=1.18}

\newcommand{\N}{\mathbb{N}}                   % natural numbers
\newcommand{\R}{\mathbb{R}}                   % real numbers
\newcommand{\Reg}{\mathrm{Reg}}               % regular locus
\newcommand{\Sing}{\mathrm{Sing}}             % singular locus
\newcommand{\Zar}{\mathrm{Zar}}               % Zariski
\newcommand{\Na}{\mathcal{N}}                 % Nash
\newcommand{\Aa}{\!\!\!\:{~}^a\:\!\!\!\mathcal{A}}% arc-analytic
\newcommand{\G}{\Gamma}
\newcommand{\Sg}{\Sigma}
\newcommand{\wt}{\widetilde}
\newcommand{\wh}{\widehat}
\newcommand{\ve}{\varepsilon}
\newcommand{\mi}{\mathfrak m}                 % maximal ideal m
\newcommand{\pp}{\mathfrak p}                 % prime ideal p
\newcommand{\rad}{\mathrm{rad}}               % radical ideal
\newcommand{\ZZ}{\mathcal Z}                  % zero set germ
\newcommand{\II}{\mathcal I}                  % ideal sheaf I
\newcommand{\vp}{\varphi}                     % mapping phi
\newcommand{\AR}{\mathcal{AR}}                % AR-
\newcommand{\Int}{\!\:I\!\!\:nt}              % interior
\newcommand{\Cl}{Cl}                          % closure
\theoremstyle{plain}
\newtheorem{theorem}{Theorem}[section]
\newtheorem{proposition}[theorem]{Proposition}
\newtheorem{lemma}[theorem]{Lemma}
\newtheorem{corollary}[theorem]{Corollary}
\newtheorem{conjecture}[theorem]{Conjecture}

\theoremstyle{definition}
\newtheorem{definition}[theorem]{Definition}
\newtheorem{example}[theorem]{Example}
\newtheorem{remark}[theorem]{Remark}
\newtheorem{question}[theorem]{Question}

\numberwithin{equation}{section}

\begin{document}

\title[Arc-analytic functions on locally closed sets]{Arc-analytic functions on locally closed\\ semialgebraic sets}

\author{Janusz Adamus}
\address{Department of Mathematics, The University of Western Ontario, London, Ontario, N6A 5B7 Canada}
\email{jadamus@uwo.ca}
\author{Matthew R. Aharonian}
\address{Department of Mathematics, The University of Western Ontario, London, Ontario, N6A 5B7 Canada}
\email{maharoni@uwo.ca}
\author{Sean Coverdale}
\address{Department of Mathematics, The University of Western Ontario, London, Ontario, N6A 5B7 Canada}
\email{scoverda@uwo.ca}
\author{Jaden Visscher}
\address{Department of Mathematics, The University of Western Ontario, London, Ontario, N6A 5B7 Canada}
\email{jvissch3@uwo.ca}
\thanks{Research was partially supported by the Natural Sciences and Engineering Research Council of Canada (Adamus, Aharonian, Visscher) and the University of Western Ontario USRI scholarship (Aharonian).}


\begin{abstract}
We give a partial generalization of the Kurdyka theory of arc-analytic functions and arc-symmetric sets on Nash manifolds to the setting of arbitrary locally closed semialgebraic sets.
\end{abstract}
\maketitle

%%%%%%%%%%%%%%%%%%%%%%%%%%%%%%%%%%%%%%%%%%%%%%%%%%
%%%%%%%%%%%%%%%%% Section %%%%%%%%%%%%%%%%%%%%%%%%
%%%%%%%%%%%%%%%%%%%%%%%%%%%%%%%%%%%%%%%%%%%%%%%%%%

\section{Introduction}
\label{sec:intro}

This article is concerned with the study of semialgebraic arc-analytic functions and the sets described as common zero-loci of collections of such functions. The notion of an arc-analytic function --- that is, a function analytic along all real analytic arcs into its domain --- was introduced by K.\,Kurdyka in his seminal 1988 paper~\cite{Kur} and soon became ubiquitous in both real algebraic and analytic geometry.
Although in general arc-analytic functions can be quite wild (e.g., arc-analyticity on its own does not even guarantee continuity, see~\cite{BMP}), in the semialgebraic setting, by contrast, they give rise to a very elegant and rich theory, which in many respects resembles the classical algebraic geometry over an algebraically closed field. 

Let $M$ be a Nash manifold, and let $\Aa(M)$ denote the ring of semialgebraic arc-analytic functions on $M$ (see Section~\ref{sec:prelim} for definitions of \emph{Nash}, \emph{semialgebraic}, and other terms used in this introduction). Let us mention here just a few of the remarkable properties of these functions, which we recall in detail in Section~\ref{sec:prelim}:
\begin{itemize}
\item[(i)] The sets defined by the common vanishing of elements of ideals in $\Aa(M)$ form the family of closed sets of a Noetherian topology, the so-called $\AR$ topology on $M$. 
\item[(ii)] The closed sets of this topology are precisely the semialgebraic arc-symmetric subsets of $M$, that is, sets that are closed under analytic arc continuation.
\item[(iii)] The $\AR$ topology on a Euclidean space $M=\R^n$ is strictly finer than the Zariski topology, and it gives rise to a decomposition of real algebraic sets into irreducible components that is in a (quite precise) sense compatible with their desingularization.
\item[(iv)] The ideals of the ring $\,\Aa(M)$ and the $\AR$-closed subsets of $M$ satisfy the arc-analytic variant of Nullstellensatz.
\item[(v)] The arc-analytic functions on $\AR$-closed subsets of $M$ admit arc-analytic extensions to the whole $M$.
\item[(vi)] The Krull dimension of the ring $\Aa(E)$ of arc-analytic functions on an $\AR$-closed set $E$ in $M$ is equal to the Euclidean dimension of $E$.
\end{itemize}

In the present paper, we give a partial generalization of the Kurdyka theory of arc-analytic functions and arc-symmetric sets from Nash manifolds to arbitrary locally closed semialgebraic sets. The local closedness condition, in fact, is the absolute minimum that needs to be assumed of a base set of our theory, for the simple reason that without it there is no {\L}ojasiewicz inequality (or, equivalently, no semialgebraic weak Nullstellensatz; cf. Proposition~\ref{prop:weak-nullstellensatz}).

Let $S\subset\R^n$ be a locally closed semialgebraic set. The main results of the paper are:
\begin{itemize}
\item[(i)] Theorem~\ref{thm:general-AR-top}, asserting that the semialgebraic arc-symmetric subsets of $S$ form the family of closed sets of a Noetherian topology, the so-called $S$-$\AR$ topology on $S$,
\item[(ii)] Theorem~\ref{thm:irred-criterion}, which shows that $S$-$\AR$-irreducibility is compatible with desingularization,
\item[(iii)] Theorem~\ref{thm:basic-closed-has-AEP}, asserting that $S$ has the arc-analytic extension property provided it is of the form
\[
S=\{x\in M: h_1(x)\geq0,\dots,h_s(x)\geq0\}
\]
for some Nash manifold $M$ and arc-analytic semialgebraic functions $h_1,\dots,h_s$ on $M$,
\item[(iv)] Proposition~\ref{prop:AEP-implies-ZSP}, asserting that every semialgebraic set with the arc-analytic extension property enjoys the zero-set property, and
\item[(v)] Theorems~\ref{thm:general-Nullstellensatz} and~\ref{thm:dim-equals-Krull-dim}, showing that the arc-analytic Nullstellensatz and equality between the Krull dimension of the ring $\Aa(S)$ and the Euclidean dimension of $S$ hold for $S$ with the zero-set property.
\end{itemize}

What makes this theory interesting is the fact that, unless $S\subset\R^n$ is an arc-symmetric subset of $\R^n$, the intrinsic $S$-$\AR$ topology on $S$ is strictly finer than the induced $\AR$-subspace topology from $\R^n$ (see Remark~\ref{rem:AR-not-induced} and Examples~\ref{ex:1reducible} and~\ref{ex:2reducible}).
Likewise, the algebra $\Aa(S)$ of arc-analytic functions on a non-arc-symmetric $S$ is not at all determined by $\Aa(\R^n)$ (see Theorem~\ref{thm:Aa-integral-equiv}).
And there are many natural examples of non-arc-symmetric locally closed semialgebraic sets, whose algebra of arc-analytic functions is of interest (for instance, the Nash sheets of arc-symmetric sets, as defined by Kurdyka~\cite[\S\,3]{Kur}, or images of compact sets by local blow-ups).
\smallskip

The structure of the paper is as follows: To keep the paper reasonably self-contained for non-specialist readers, in the next section we provide basic definitions and give a review of the classical theory of arc-analytic functions and arc-symmetric sets on Nash manifolds. Section~\ref{sec:S-AR-topology} contains a discussion of arc-analytic functions and arc-symmetric sets on arbitrary semialgebraic sets. In Section~\ref{sec:arc-irreducibility}, we introduce a notion of irreducibility of a semialgebraic set $S$ in terms of the algebra $\Aa(S)$ of its arc-analytic functions, which we call \emph{arc-analytic irreducibility}, and study its basic properties. Section~\ref{sec:irred-comps} is devoted to the study of irreducible components of a semialgebraic set $S$. In Section~\ref{sec:arcan-null}, we prove our general arc-analytic Nullstellensatz and its algebro-geometric consequences. In the final section, we prove Theorem~\ref{thm:basic-closed-has-AEP} and its consequences. Chief among them, Proposition~\ref{prop:AEP-implies-ZSP} which asserts that the arc-analytic extension property implies the zero-set property.
\medskip

We are happy to acknowledge the influence of the unpublished preprint~\cite{S} by H. Seyedinejad on our work.
In fact, to the best of our knowledge, Seyedinejad's note was the first attempt at describing the ring of arc-analytic functions on a general semialgebraic set and its relation to the geometry of the set, and the original goal of our project was to understand and correct several subtle errors in his exposition. This has lead to our distillation of the zero-set property and the arc-analytic extension property (Definition~\ref{def:zero-set-arcan-ext}) as the central objects of the study.

\vspace{2ex}
%%%%%%%%%%%%%%%%%%%%%%%%%%%%%%%%%%%%%%%%%%%%%%%%%%
%%%%%%%%%%%%%%%%% Section %%%%%%%%%%%%%%%%%%%%%%%%
%%%%%%%%%%%%%%%%%%%%%%%%%%%%%%%%%%%%%%%%%%%%%%%%%%

\section{Preliminaries}
\label{sec:prelim}

Recall that a set $S$ in $\R^n$ is called \emph{semialgebraic}, if it can be written as a finite union of sets of the form
\[
\{x\in\R^n\,:\,f(x)=0,\,g_1(x)>0,\dots,g_s(x)>0\}\,,
\]
for some $s\in\N$, where $f,g_1,\dots,g_s\in\R[x]$ and $x=(x_1,\dots,x_n)$. Given $S\subset\R^n$, a \emph{semialgebraic function} $f:S\to\R$ is one whose graph $\G_f$ is a semialgebraic subset of $\R^{n+1}$. It follows from the celebrated Tarski-Seidenberg theorem that the domain and level sets of a semialgebraic function are semialgebraic sets. A real analytic submanifold $M$ in $\R^n$ (not necessarily closed) is called a \emph{Nash manifold}, if $M$ is a semialgebraic subset of $\R^n$.
A function $f:U\to\R$ on an open semialgebraic set $U\subset\R^n$ is a \emph{Nash function}, if $f$ is real analytic and semialgebraic. We write $\Na(U)$ for the ring of all Nash functions on $U$. For an excellent detailed exposition of semialgebraic geometry, we refer the reader to \cite{BCR}. Here, we shall only mention two results that are used frequently throughout the paper.

\begin{proposition}[\textbf{Curve Selection Lemma}]
\label{prop:Nash-curve-sel}
Let $M$ be a Nash manifold, let $A\subset M$ be a semialgebraic set, and let $a\in\Cl_M(A)$. There exists a Nash mapping $\vp:(-1,1)\to M$ such that $\vp((0,1))\subset A$ and $\vp(0)=a$.
\end{proposition}

A \emph{semialgebraic} (or \emph{Nash}) \emph{stratification} of a semialgebraic set $S\subset\R^n$ is a finite partition $\{S_\alpha\}_{\alpha\in A}$ of $S$, where each $S_\alpha$ is a connected Nash submanifold of $\R^n$ such that, if $S_\alpha\cap(\Cl(S_\beta)\setminus S_\beta)\neq\varnothing$ then $S_\alpha\subset\Cl(S_\beta)\setminus S_\beta$ and $\dim{S_\alpha}<\dim{S_\beta}$. (Here and throughout, $\Cl(S)$ denotes the Euclidean closure of a set $S\subset\R^n$.)

\begin{proposition}[\textbf{Nash Stratification}]
\label{subordinate-strat}
Let $S_1,\dots,S_k$ be semialgebraic sets in $\R^n$. There exists a Nash stratification of $\R^n$ such that each $S_j$ is a union of some strata of this stratification.
\end{proposition}
\smallskip

\subsection*{Semialgebraic arc-analytic functions and arc-symmetric sets}

We recall below basic facts and most important results of the theory of semialgebraic arc-analytic functions and arc-symmetric sets, that we shall use and refer to in the paper. Although originally the theory was constructed for subsets of a Euclidean space $\R^n$, it can easily be transferred to Nash manifolds thanks to the fact that every Nash manifold is Nash isomorphic to a nonsingular algebraic set in $\R^N$, for some $N$, by \cite[Thm.\,VI.2.1, Rem.\,VI.2.11]{Shiota} (cf.~\cite[\S\,1]{KK}). Since one of the main techniques used in this paper is the blow-up of a Nash manifold, the Nash manifold setting will be more useful for our purposes, as it will allow us to remain within our initial category after blowing-up.

\begin{definition}[{\cite[Def.\,1.1]{Kur}}]
\label{def:arc-symm}
Let $M$ be a Nash manifold. A semialgebraic set $E\subset M$ is called \emph{arc-symmetric} if it satisfies one of the following equivalent conditions:
\begin{itemize}
\item[(i)] For every analytic arc $\gamma:(-1,1)\to M$, if $\Int(\gamma^{-1}(E))\neq\varnothing$ then $\gamma^{-1}(E)=(-1,1)$
\item[(ii)] For every analytic arc $\gamma:(-1,1)\to M$, if $\gamma((-1,0))\subset E$ then $\gamma((-1,1))\subset E$
\item[(iii)] For every injective analytic arc $\gamma:(-1,1)\to M$, if $\gamma((-1,0))\subset E$ then $\gamma((-1,1))\subset E$.
\end{itemize}
The equivalence of conditions (i) and (ii) with (iii) follows from \cite[Lem.\,0.1]{Kur}. Here and throughout, $\Int(S)$ denotes the (Euclidean) interior of a set $S$.
\end{definition}

For the following theorem, recall that a topological space $(X,\tau)$ is called \emph{Noetherian}, when for every descending chain $F_1\supset F_2\supset F_3\supset\dots$ of closed sets in $X$ there exists $n$ such that $F_n=F_{n+k}$ for all $k\in\N$.

\begin{theorem}[{\cite[Thm.\,1.4]{Kur}}]
\label{thm:AR-topology}
Let $M$ be a connected Nash manifold.
There exists a Noetherian topology on $M$, whose closed sets are precisely the arc-symmetric semialgebraic subsets of $M$.
\end{theorem}

The proof of Theorem~\ref{thm:AR-topology} boils down to establishing the following two lemmas.

\begin{lemma}[{\cite[1.6]{Kur}}]
\label{lem:dim-drop}
Let $\G$ be a connected smooth semialgebraic subset of $M$, and let $E\subset M$ be arc-symmetric. Then
\[
\G\not\subset E\ \Longrightarrow\ \dim(\G\cap E)<\dim\G\,.
\]
\end{lemma}

\begin{lemma}[{\cite[Lem.\,1.5]{Kur}}]
\label{lem:fin-intersection}
Let $\G$ be a connected smooth semialgebraic subset of $M$, and let $\{E_i\}_{i\in I}$ be arc-symmetric subsets of $M$. Then, there exist $i_1,\dots,i_s\in I$ such that
\[
\G\cap\bigcap_{i\in I}E_i\ =\ \G\cap E_{i_1}\cap\dots\cap E_{i_s}\,.
\]
\end{lemma}

By Noetherianity of the $\AR$-topology, for every semialgebraic set $E$ in a Nash manifold $M$ there exists a unique smallest $\AR$-closed set in $M$ containing $E$, called its \emph{$\AR$-closure}.
The following observation follows from the fact that taking Zariski closure of a semialgebraic set does not increase the dimension (\cite[Prop.\,2.8.2]{BCR}). Here and throughout the paper, we write $\Cl^\tau_M(S)$ to denote the closure of a set $S$ in a topological space $(M,\tau)$; if $\tau$ is the Euclidean topology, we write simply $\Cl_M(S)$.

\begin{proposition}
\label{prop:alg-closure-dim}
Let $E$ be an arbitrary semialgebraic set in a Nash manifold $M\subset\R^n$. Let $\Cl^\AR_M(E)$ denote the smallest $\AR$-closed subset of $M$ containing $E$,
let $\Cl^\Na_M(E)$ denote the smallest Nash subset of $M$ containing $E$, and let $\Cl^\Zar(E)$ be the smallest algebraic set in $\R^n$ containing $E$. Then,
\[
\dim{E}=\dim\Cl^\AR_M(E)=\dim\Cl^\Na_M(E)=\dim\Cl^\Zar(E)\,.
\]
\end{proposition}

The following result is the most consequential for the classical theory (see below for a definition of resolution of singularities).

\begin{theorem}[{\cite[Thm.\,2.6]{Kur}}]
\label{thm:AR-smoothing}
Let $X$ be a Nash set in a Nash manifold $M$, of dimension $k>0$, and let $\pi:\wt{X}\to X$ be its resolution of singularities by blow-ups with smooth Nash centres. Let $E$ be an $\AR$-irreducible subset of $X$ of dimension $k$. Then, there exists a unique connected component $\wt{E}$ of $\wt{X}$ such that
\[
\pi(\wt{E})=\Cl_M(\Reg_k{E})\,.
\]
\end{theorem}

\begin{corollary}[{\cite[Cor.\,2.8]{Kur}}]
\label{cor:1-AR-irred}
Let $M$ be a Nash manifold, and let $E\subset M$ be an $\AR$-closed set of dimension $k$. Let $C_1,\dots,C_s$ be the connected components of $\Reg_k{E}$. The following conditions are equivalent:
\begin{itemize}
\item[(i)] $E$ has precisely one $\AR$-irreducible component of dimension $k$.
\item[(ii)] For all $p,q\in\Cl_M(\Reg_k{E})$, there exists an arc $\gamma:[0,1]\to\Cl_M(\Reg_k{E})$ analytic in an open neighbourhood of $[0,1]$, such that
$\gamma(0)=p$ and $\gamma(1)=q$\,.
\item[(iii)] For every $j$, there exists a nonempty open subset $S_j$ in $C_j$ such that, for all points $p_j\in S_j$ there are analytic arcs $\gamma_j:[0,1]\to\Cl_M(\Reg_k{E})$ with
$\gamma_j(0)=p_j$ and $\gamma_j(1)=p_{j+1}$ ($j=1,\dots,s-1$).
\end{itemize}
\end{corollary}

Another critical corollary to Theorem~\ref{thm:AR-smoothing} is the arc-analytic Identity Principle:

\begin{proposition}[{\cite[Prop.\,5.3]{Kur}}]
\label{prop:ident-princ}
Let $X$ be an $\AR$-irreducible set in a Nash manifold $M$, and let $f:X\to\R$ be a semialgebraic arc-analytic function on $X$ such that $f|_U=0$ for a semialgebraic set $U$ with $\dim{U}=\dim{X}$. Then, $f\equiv0$.
\end{proposition}

More recently, the first named author together with H.\,Seyedinejad proved that the $\AR$-closed sets are precisely the zero-loci of arc-analytic functions, and consequently that Nullstellensatz holds in the arc-analytic category.

\begin{theorem}[{\cite[Thm.\,1]{ASe1}}]
\label{thm:ASey-zero-set-thm}
Let $X$ be an $\AR$-closed set in a Nash manifold $M$. There exists a semialgebraic arc-analytic function $f:M\to\R$ such that $X=f^{-1}(0)$.
\end{theorem}

\begin{theorem}[{\cite[Prop.\,1]{ASe1}}]
\label{thm:AR-Nullstellensatz}
Let $X$ be an $\AR$-closed set in a Nash manifold $M$.
Let $\II_X(Y)$ denote the vanishing ideal in $\Aa(X)$ of $Y\subset X$, and let $\ZZ_X(I)$ denote the zero-set of an ideal $I$ in $\Aa(X)$. Then:
\begin{itemize}
\item[(i)] If \,$Y\subset X$ is $\AR$-closed, then $\ZZ_X(\II_X(Y))=Y$.
\item[(ii)] If $I$ is an ideal in $\Aa(X)$, then $\II_X(\ZZ_X(I))=\rad_X(I)$.
\end{itemize}
\end{theorem}

Lastly, arc-analytic functions on $\AR$-closed sets in $M$ are but restrictions of arc-analytic functions on $M$.

\begin{theorem}[{\cite[Thm.\,1]{ASe2}}]
\label{thm:ASey-arc-an-extension}
Let $X$ be an $\AR$-closed set in a Nash manifold $M$, and let $f:X\to\R$ be an arc-analytic semialgebraic function on $X$. Then, there exists an arc-analytic semialgebraic function $F:M\to\R$ such that $F|_X=f$. In other words,
\[
\Aa(X) \,\cong \,\Aa(M)/\II_M(X)
\]
as $\R$-algebras.
\end{theorem}
\smallskip

\subsection*{Resolution of Singularities}

One of the main techniques used in the present paper is that of embedded desingularization. The celebrated theorem of H.\,Hironaka~\cite{H} (see also~\cite{BM} and~\cite{W}) asserts that every real algebraic set has a smooth model, which is isomorphic to the given set outside a subset of strictly smaller dimension.

More precisely, let $X$ be an analytic set in an analytic manifold $M$. Consider a sequence of transformations
\begin{align}
\label{eq:desing}
\notag
&\cdots\ \longrightarrow &M_{j+1} &\qquad\stackrel{\pi_{j+1}}{\longrightarrow} &M_j &\qquad\longrightarrow\ \cdots\ \longrightarrow &M_1 &\qquad\stackrel{\pi_1}{\longrightarrow} &M_0&=M\\
&{~}             &X_{j+1} &{~}                                      &X_j &{~}                                    &X_1 &{~}                                  &X_0&=X\\
\notag
&{~}  					 &E_{j+1} &{~}                                      &E_j &{~}                                    &E_1 &{~}                                  &E_0&=\varnothing
\end{align} 
where, for every $j$, $\pi_{j+1}:M_{j+1}\to M_j$ denotes a blow-up of $M_j$ with a smooth centre $C_j\subset M_j$, $X_{j+1}$ is the strict transform of $X_j$ by $\pi_{j+1}$, and $E_{j+1}$ denotes the exceptional hypersurfaces (that is, $E_{j+1}$ is the union of strict transforms of all hypersurfaces $H\subset E_j$, together with $\pi^{-1}_{j+1}(C_j)$).
\medskip

\begin{theorem}[\textbf{Embedded desingularization}]
\label{thm:res-alg}
Every algebraic set $X$ in $M=\R^n$ admits an embedded desingularization. That is, there is a proper mapping $\pi:\wt{M}\to M$, which is a composition of a finite sequence~\eqref{eq:desing} of blow-ups with smooth algebraic centers, such that $\pi$ is an isomorphism outside the preimage of the (algebraic) singular locus $\Sing{X}$ of $X$, the final strict transform $\wt{X}$ of $X$ is smooth, and $\wt{X}$ and $\pi^{-1}(\Sing{X})$ simultaneously have only normal crossings. (The latter means that every point of $\wt{M}$ admits a (local analytic) coordinate neighbourhood in which $\wt{X}$ is a coordinate subspace and each hypersurface $H$ of $\pi^{-1}(\Sing{X})$ is a coordinate hypersurface.)
\end{theorem}

Since every Nash manifold is a closed semialgebraic subset in some $\R^N$, Theorem~\ref{thm:res-alg} implies the following Nash variant of resolution.

\begin{corollary}
\label{cor:res-Nash}
Let $X$ be a $k$-dimensional Nash set in a Nash manifold $M$. There exists a proper Nash mapping $\pi:\wt{M}\to M$ of Nash manifolds, which is a composition of a finite sequence~\eqref{eq:desing} of blow-ups with smooth Nash centers, such that $\pi$ is a Nash isomorphism outside the preimage of a Nash subset $Y\subset M$ of dimension $\dim{Y}<k$, the final strict transform $\wt{X}$ of $X$ is a $k$-dimensional Nash manifold, and $\wt{X}$ and $\pi^{-1}(Y)$ simultaneously have only normal crossings.
\end{corollary}

\begin{remark}
\label{rem:exc-locus}
Given a Nash manifold $M$ and a finite composite $\pi=\pi_1\circ\dots\circ\pi_s:M_s\to M_0=M$ of blow-ups $\pi_{j+1}:M_{j+1}\to M_j$ with smooth Nash centers $C_j\subset M_j$, the set 
\[
E=C_0\cup\pi_1(C_1)\cup(\pi_1\circ\pi_2)(C_2)\cup\dots\cup(\pi_1\circ\dots\circ\pi_{s-1})(C_{s-1})
\]
is called the \emph{exceptional locus} of $\pi$. If $D$ denotes the exceptional divisor of $\pi$, then $D=\pi^{-1}(E)$ and $\pi$ is actually an isomorphism over $M\setminus E$. Note that, if $Y\subset M$ is the Nash set from Corollary~\ref{cor:res-Nash} then, in general, one only has $E\subset Y$ but the inclusion may be proper. Nonetheless, it is an easy exercise to show that $E$ is always $\AR$-closed in $M$.
\end{remark}

\vspace{2ex}
%%%%%%%%%%%%%%%%%%%%%%%%%%%%%%%%%%%%%%%%%%%%%%%%%%
%%%%%%%%%%%%%%%%% Section %%%%%%%%%%%%%%%%%%%%%%%%
%%%%%%%%%%%%%%%%%%%%%%%%%%%%%%%%%%%%%%%%%%%%%%%%%%

\section{Arc-analytic functions and $\AR$ topology on general semialgebraic sets}
\label{sec:S-AR-topology}

Let $S$ be an arbitrary semialgebraic set in $\R^n$.
We define arc-analytic functions on $S$ in an obvious fashion.

\begin{definition}
\label{def:arc-an-on-semialg}
A function $f:S\to\R$ is called \emph{arc-analytic}, if $f\circ\gamma$ is real analytic for every analytic arc $\gamma:(-1,1)\to S$. The $\R$-algebra of all semialgebraic arc-analytic functions on $S$ will be denoted by $\Aa(S)$.
\end{definition}

By Theorem~\ref{thm:ASey-arc-an-extension}, the algebra of arc-analytic functions on an arc-symmetric subset $S$ of a Nash manifold $M$ is completely determined by the algebra $\Aa(M)$. For a general semialgebraic set $S$ however, $\Aa(S)$ has a much richer structure, as implied by the following result.

\begin{theorem}
\label{thm:Aa-integral-equiv}
Let $S$ be a semialgebraic subset of a Nash manifold $M$. Let $\varrho:\Aa(M)\to\Aa(S)$ be the restriction homomorphism. The following conditions are equivalent:
\begin{itemize}
\item[(i)] $S$ is an arc-symmetric subset of $M$
\item[(ii)] $\varrho$ is surjective
\item[(iii)] $\varrho$ makes $\Aa(S)$ into a finite $\Aa(M)$-module
\item[(iv)] $\varrho$ makes $\Aa(S)$ into an integral $\Aa(M)$-algebra.
\end{itemize}
\end{theorem}

\begin{proof}
The implication (i)$\,\Rightarrow\,$(ii) follows from Theorem~\ref{thm:ASey-arc-an-extension}. Implications (ii)$\,\Rightarrow\,$(iii) and (iii)$\,\Rightarrow\,$(iv) are true for any ring homomorphism. We prove (iv)$\,\Rightarrow\,$(i) by contradiction. Suppose $S$ is not arc-symmetric but $\Aa(S)$ is integral over $\Aa(M)$. Let $\Cl^\AR_M(S)=S_1\cup\dots\cup S_r$ be the decomposition of the $\AR$-closure of $S$ in $M$ into $\AR$-irreducible components. Then, there exists a component, say $S_1$, such that $S_1\not\subset S$. Pick $x_0\in S_1\setminus S$. Let $d^2_{x_0}$ denote the squared distance from $x_0$ function. Then, $1/d^2_{x_0}$ is in $\Aa(S)$, and hence is integral over $\Aa(M)$. That is, there exist $n\geq1$ and $a_0,\dots,a_{n-1}\in\Aa(M)$ such that
\[
\left(\frac{1}{d^2_{x_0}}\right)^n+\left(\frac{1}{d^2_{x_0}}\right)^{n-1}\cdot a_{n-1}+\dots+\frac{1}{d^2_{x_0}}\cdot a_1+a_0=0
\]
on $S$. Hence,
\[
1+d^2_{x_0}\cdot a_{n-1}+\dots+d^{2(n-1)}_{x_0}\cdot a_1+d^{2n}_{x_0}\cdot a_0=0
\]
on $S$. Observe that the left hand side of the above is semialgebraic and arc-analytic on $M$ and, by Proposition~\ref{prop:alg-closure-dim} and Lemma~\ref{lem:Noeth-decomp} below, $\dim(S\cap S_1)=\dim{S_1}$. Thus, by the Identity Principle (Proposition~\ref{prop:ident-princ}), $1+d^2_{x_0}\cdot a_{n-1}+\dots+d^{2(n-1)}_{x_0}\cdot a_1+d^{2n}_{x_0}\cdot a_0=0$ on all of $S_1$. Evaluating this equation at $x_0$ yields $1=0$; a contradiction.
\end{proof}

\begin{lemma}
\label{lem:Noeth-decomp}
Let $S$ be a nonempty set in a Noetherian topological space $X$. Let $\Cl_X(S)=S_1\cup\dots\cup S_t$ be the (unique non-redundant) decomposition of its closure into irreducible closed sets. Then, $\Cl_X(S\cap S_i)=S_i$ for all $i=1,\dots,t$.
\end{lemma}

\begin{proof}
Without loss of generality suppose that $i=1$. Clearly, $\Cl_X(S\cap S_1)\subset S_1$, and hence $\Cl_X(S\cap S_1)\cup S_2\cup\dots\cup S_t\subset\Cl_X(S)$. In fact, since the left hand side contains $S$, we have equality. Uniqueness of a non-redundant decomposition into irreducible closed sets now implies that $S_1\subset\Cl_X(S\cap S_1)$.
\end{proof}
\medskip

Next, we define arc-symmetric subsets of an arbitrary semialgebraic set $S\subset\R^n$, as follows.

\begin{definition}
\label{def:arc-symm-in-semialg}
A subset $E\subset S$ is called \emph{arc-symmetric in $S$}, if it satisfies one of the following equivalent conditions:
\begin{itemize}
\item[(i)] For every analytic arc $\gamma:(-1,1)\to S$, $\Int(\gamma^{-1}(E))\neq\varnothing$ implies $\gamma^{-1}(E)=(-1,1)$
\item[(ii)] For every analytic arc $\gamma:(-1,1)\to S$, $\gamma((-1,0))\subset E$ implies $\gamma((-1,1))\subset E$
\item[(iii)] For every injective analytic arc $\gamma:(-1,1)\to S$, $\gamma((-1,0))\subset E$ implies $\gamma((-1,1))\subset E$.
\end{itemize}
\end{definition}

A key part in Kurdyka's theory recalled in Section~\ref{sec:prelim} is played by Noetherianity of the $\AR$-topology on a Nash manifold.
An analogous property is actually enjoyed by all semialgebraic sets, since in essence it relies only on basic topological properties of o-minimal structures (cf.~\cite{vDD}).

\begin{theorem}
\label{thm:general-AR-top}
Let $S\subset\R^n$ be a semialgebraic set. There is a Noetherian topology on $S$, whose closed sets coincide with the semialgebraic arc-symmetric subsets of $S$.
\end{theorem}

We shall prove Theorem~\ref{thm:general-AR-top} in the same way as~\cite[Thm.\,1.4]{Kur}. To this end, we will need the following two lemmas.

\begin{lemma}
\label{lem:AR-closed-is-closed}
Let $S\subset\R^n$ be a semialgebraic set, and let $E\subset S$ be a semialgebraic subset which is arc-symmetric in $S$.
Then, $E$ is closed in the Euclidean topology on $S$.
\end{lemma}

\begin{proof}
Let $S$ and $E$ be as above, and let $p\in\Cl_S(E)$ be arbitrary. By the Curve Selection Lemma (Proposition~\ref{prop:Nash-curve-sel}), there is a Nash mapping $\gamma:(-1,1)\to\R^n$ such that $\gamma(0)=p$ and $\gamma((0,1))\subset E$. We may assume that $\gamma((-1,1))\subset S$, by composing $\gamma$ with $\{t\mapsto t^2\}$ if needed. Then, by arc-symmetry of $E$, $\gamma((-1,1))\subset E$, and hence $p=\gamma(0)\in E$.
\end{proof}

\begin{lemma}
\label{lem:general-finite-intersection}
Let $S\subset\R^n$ be a semialgebraic set, and let $\{E_i\}_{i\in I}$ be an arbitrary family of semialgebraic arc-symmetric subsets of $S$. There exists a finite index subset $I_0=\{i_1,\dots,i_t\}\subset I$ such that
\[
\bigcap_{i\in I}E_i\;=\;E_{i_1}\cap\dots\cap E_{i_t}\,.
\]
\end{lemma}

\begin{proof}
Let $S=S_1\cup\dots\cup S_r$ be a semialgebraic stratification of $S$. For every $j=1,\dots,r$, by Lemma~\ref{lem:fin-intersection} applied to $\G=M=S_j$ and the family $\{E_i\cap S_j\}_{i\in I}$ of arc-symmetric subsets of $S_j$, there is a finite index set $I_j\subset I$ such that
\[
S_j\cap\bigcap_{i\in I}E_i\;=\;S_j\cap\bigcap_{i\in I_j}E_i\,.
\]
Then, $I_0=I_1\cup\dots\cup I_r$ has the required property.
\end{proof}

\begin{proof}[Proof of Theorem~\ref{thm:general-AR-top}]
Let $\mathcal{F}$ be the collection of semialgebraic arc-symmetric subsets of $S$. Clearly, both $\varnothing$ and $S$ belong to $\mathcal{F}$.
Given $E_1,\dots,E_t\in\mathcal{F}$, to show that $E_1\cup\dots\cup E_t\in\mathcal{F}$, it suffices to consider the case when $t=2$. Let $\gamma:(-1,1)\to S$ be an analytic arc with $\Int(\gamma^{-1}(E_1\cup E_2))\neq\varnothing$. By Lemma~\ref{lem:AR-closed-is-closed} and continuity of $\gamma$, both $\gamma^{-1}(E_1)$ and $\gamma^{-1}(E_2)$ are closed subsets of $(-1,1)$. Hence, $\Int(\gamma^{-1}(E_i))\neq\varnothing$ for some $i=1,2$, for otherwise
\[
\Int(\gamma^{-1}(E_1)\cup\gamma^{-1}(E_2))=\Int(\gamma^{-1}(E_1\cup E_2))\neq\varnothing\,,
\]
contradicting the Baire Category Theorem. Say, $\Int(\gamma^{-1}(E_1))\neq\varnothing$. Then, by arc-symmetry of $E_1$, $\gamma((-1,1))\subset E_1\subset E_1\cup E_2$. Hence, $E_1\cup E_2\in\mathcal{F}$.

Finally, let $\{E_i\}_{i\in I}\subset\mathcal{F}$ be arbitrary. To show that $\bigcap_i E_i\in\mathcal{F}$, by Lemma~\ref{lem:general-finite-intersection} we may assume that $I$ is finite, say, $I=\{i_1,\dots,i_t\}$. Let $\gamma:(-1,1)\to S$ be an analytic arc with $\Int(\gamma^{-1}(E_{i_1}\cap\dots\cap E_{i_t}))\neq\varnothing$. Then,
\[
\Int(\gamma^{-1}(E_{i_1}))\cap\dots\cap\Int(\gamma^{-1}(E_{i_t}))=\Int(\gamma^{-1}(E_{i_1})\cap\dots\cap\gamma^{-1}(E_{i_t}))=\Int(\gamma^{-1}(E_{i_1}\cap\dots\cap E_{i_t}))\neq\varnothing\,.
\]
Hence, by arc-symmetry of each $E_{i_j}$, $\gamma((-1,1))\subset E_{i_j}$ for all $j$. Thus, $\bigcap_{i\in I}E_i=E_{i_1}\cap\dots\cap E_{i_t}\in\mathcal{F}$.
Noetherianity of this topology follows again from Lemma~\ref{lem:general-finite-intersection}, since every descending chain of sets in $\mathcal{F}$ stabilizes.
\end{proof}

\begin{definition}
\label{def:S-AR-top}
For a semialgebraic set $S$, we shall call the above topology on $S$ the \emph{$S$-$\AR$ topology}, and its closed sets the \emph{$S$-$\AR$-closed} subsets of $S$.
\end{definition}

\begin{remark}
\label{rem:AR-not-induced}
In general, the $S$-$\AR$ topology on a semialgebraic set $S\subset\R^n$ is strictly finer than the induced $\AR$-subspace topology from $\R^n$ (or from any Nash manifold $M$ containing $S$). Indeed, a trivial example which demonstrates this fact is $S=\R\setminus\{0\}$ as a subset of $M=\R$. The only $\AR$-closed subsets of $\R$ are finite sets and $\R$ itself. Hence, the closed subsets of the $\AR$-subspace topology on $S$ are exactly the finite subsets of $S$ and $S$ itself. However, as is easy to see, both $(-\infty,0)$ and $(0,\infty)$ are $S$-$\AR$-closed. See Examples~\ref{ex:1reducible} and~\ref{ex:2reducible} below for much less obvious (and surprising) specimens.
\end{remark}

By Noetherianity of the $S$-$\AR$ topology, for every subset $E\subset S$ there exists a unique smallest $S$-$\AR$-closed set containing $E$, called its \emph{$S$-$\AR$-closure}, which we denote by $\Cl^\AR_S(E)$.
A nonempty $S$-$\AR$-closed set $F$ is called \emph{$S$-$\AR$-irreducible}, if $F$ is not a union of two nonempty proper $S$-$\AR$-closed subsets.
As in any Noetherian topological space, for every nonempty $S$-$\AR$-closed set $F$ there is a unique finite collection of $S$-$\AR$-irreducible sets $F_1,\dots,F_s$, called the \emph{$S$-$\AR$-irreducible components} of $F$, such that $F=F_1\cup\dots\cup F_s$ and, for all $j=1,\dots,s$,
\[
F_j\not\subset\bigcup_{k\neq j}F_k\,.
\]
\bigskip

Kurdyka's \cite[Prop.\,5.1]{Kur} easily generalizes to our setting. We include the proof below for completeness.

\begin{proposition}[cf.\,{\cite[Prop.\,5.1]{Kur}}]
\label{prop:arcan-AR-continuous}
Let $S\subset\R^n$, and let $f:S\to\R^m$ be a semialgebraic arc-analytic function. Then:
\begin{itemize}
\item[(i)] The graph $\G_f$ of $f$ is an arc-symmetric subset of $S\times\R^m$
\item[(ii)] $f$ is continuous in the $S$-$\AR$-topology
\item[(iii)] $f$ is continuous in the Euclidean topology.
\end{itemize}
\end{proposition}

\begin{proof}
For (i), let $\gamma:(-1,1)\to S\times\R^m$ be an analytic arc and suppose $\gamma((-1,0))\subset\G_f$. Write $\gamma=(\gamma_1,\gamma_2)$. Then, for all $t\in(-1,0)$, $\gamma_2(t)=f\circ\gamma_1(t)$. By the Identity Principle for analytic functions, it follows that $\gamma_2\equiv f\circ\gamma_1$. Therefore, $\gamma((-1,1))\subset\G_f$.

For (ii), let $E\subset\R^m$ be $\AR$-closed. Semialgebraicity of $f^{-1}(E)$ follows from the Tarski-Seidenberg theorem. Arc-symmetry of $f^{-1}(E)$ is obvious.

As for Euclidean-continuity of $f$, (i) shows that the graph of $f$ is closed in $S\times\R^m$. We also claim that $f$ is locally bounded. Suppose $f$ were not locally bounded, say at $x_0\in S$. Then, there is a sequence $(x_n)_{n=1}^\infty$ in $S$ such that $\lim_{n\to\infty}x_n=x_0$ and $\mathrm{dist}(f(x_n),f(x_0))>1$ for every $n\geq1$. The set $A=\{x\in S:\mathrm{dist}(f(x),f(x_0))>1\}$ is a semialgebraic subset of $S$, and by the above $x_0\in\Cl_S(A)$. Hence, by the Curve Selection Lemma, there is a Nash mapping $\gamma:(-1,1)\to S$ such that $\gamma((-1,0))\subset A$ and $\gamma(0)=x_0$. As $f\circ\gamma$ is continuous, we have $\lim_{t\to0^-}f\circ\gamma(t)=f(x_0)$, which is impossible since $\mathrm{dist}(f(\gamma(t)),f(x_0))>1$ for all $t<0$. Therefore, $f$ is locally bounded, and having a closed graph, $f$ is continuous, proving (iii).
\end{proof}
\smallskip

Next, we isolate two properties that will play a central role in our study of rings of arc-analytic functions on general semialgebraic sets.

\begin{definition}
\label{def:zero-set-arcan-ext}
Let $S\subset\R^n$ be a semialgebraic set. We say that:
\begin{itemize}
\item[(i)] $S$ has the \emph{zero-set property} (or, \emph{ZSP}, for short), if for every $S$-$\AR$-closed set $E\subset S$, there is a function $f\in\Aa(S)$ such that $f^{-1}(0)=E$.
\item[(ii)] $S$ has the \emph{arc-analytic extension property} (or, \emph{AEP}, for short), if for every $S$-$\AR$-closed set $E\subset S$ and for every function $f\in\Aa(E)$, there exists a function $F\in\Aa(S)$ such that $F|_E=f$.
\end{itemize}
\end{definition}

\begin{remark}
\label{rem:zero-set-arcan-ext}
\begin{itemize}
\item[(i)] From Proposition~\ref{prop:arcan-AR-continuous}(ii), it follows that for a semialgebraic set $S$ and an arc-analytic function $f\in\Aa(S)$, the fibre $f^{-1}(0)\subset S$ is an $S$-$\AR$-closed set. If $S$ has the ZSP, then all $S$-$\AR$-closed sets arise in this way. By Theorem~\ref{thm:ASey-zero-set-thm}, the ZSP holds for $\AR$-closed subsets of Nash manifolds.
\item[(ii)] As we show in Proposition~\ref{prop:AEP-implies-ZSP} below, every semialgebraic set with the AEP also satisfies the ZSP. A large class of such sets is provided by Theorem~\ref{thm:basic-closed-has-AEP}, Proposition~\ref{prop:comps-AEP-implies-AEP}, and Corollary~\ref{cor:1-dim-has-AEP} below.
\item[(iii)] For a semialgebraic set $S$, having the AEP is equivalent to the restriction homomorphism $\varrho:\Aa(S)\to\Aa(E)$ being surjective, and hence inducing an isomorphism
\[
\Aa(E)\,\cong\,\Aa(S)/\II_S(E)\,,
\]
for every $S$-$\AR$-closed set $E$. In general, as we showed in Theorem~\ref{thm:Aa-integral-equiv}, if $E\subset S$ is not arc-symmetric in $S$ then the $\Aa(S)$-module structure of $\Aa(E)$ is far more complicated.
\end{itemize}
\end{remark}

For the remainder of the paper, we shall adopt the following notation: For an ideal $I$ in the ring $\Aa(S)$ of semialgebraic arc-analytic functions on $S$, we denote by $\ZZ_S(I)$ the \emph{zero-set of $I$}, that is, the subset of $S$ of common zeroes of all the elements of $I$. For a subset $E\subset S$, we denote by $\II_S(E)$ the \emph{vanishing ideal of $E$}, that is, the ideal in $\Aa(S)$ of all the functions vanishing on $E$.

\vspace{2ex}
%%%%%%%%%%%%%%%%%%%%%%%%%%%%%%%%%%%%%%%%%%%%%%%%%%
%%%%%%%%%%%%%%%%% Section %%%%%%%%%%%%%%%%%%%%%%%%
%%%%%%%%%%%%%%%%%%%%%%%%%%%%%%%%%%%%%%%%%%%%%%%%%%

\section{Arc-analytic irreducibility of semialgebraic sets}
\label{sec:arc-irreducibility}

We now turn to the investigation of the geometric structure of a semialgebraic set from the point of view of its arc-analytic function space. This idea and many of the results of this section and the next two are due to Seyedinejad~\cite{S}.

\begin{definition}[{\cite[Def.\,3.3]{S}}]
\label{def:arcan-irreducible}
Let $S\subset\R^n$ be a semialgebraic set. We say that $S$ is \emph{irreducible} if $\Aa(S)$ is an integral domain.
\end{definition}

\begin{proposition}
\label{prop:s-AR-irred-implies-arcan-irred}
Let $S\subset\R^n$ be semialgebraic. Suppose $X\subset S$ is $S$-$\AR$-irreducible. Then $X$ is an irreducible semialgebraic set.
\end{proposition}

\begin{proof}
Indeed, suppose $f_1,f_2\in\Aa(X)$ are such that $f_1f_2=0$.
Both $\ZZ_X(f_1)$ and $\ZZ_X(f_2)$ are arc-symmetric subsets of $X$ and hence, as $X$ is arc-symmetric in $S$, both $\ZZ_X(f_1)$ and $\ZZ_X(f_2)$ are arc-symmetric in $S$. Thus, $X=\ZZ_X(f_1)\cup\ZZ_X(f_2)$ is a decomposition into $S$-$\AR$-closed sets and so by $S$-$\AR$-irreducibility of $X$, $\ZZ_X(f_i)=X$ for some $i=1,2$.
This means $f_1=0$ or $f_2=0$.
\end{proof}

\begin{proposition}[cf.~{\cite[Prop.\,3.8]{S}}]
\label{prop:prime-when-irred}
Let $S\subset\R^n$ be a semialgebraic set and let $E\subset S$ be its semialgebraic subset.
If $E$ is irreducible, then $\II_S(E)$ is a prime ideal in $\Aa(S)$.
The converse holds if $S$ has the arc-analytic extension property and $E$ is $S$-$\AR$-closed.
\end{proposition}

\begin{proof}
The first claim is easy to see.
The second one follows from the ring isomorphism in Remark~\ref{rem:zero-set-arcan-ext}(iii).
\end{proof}

The following lemma expands on~\cite[Lem.\,3.4]{S}.

\begin{lemma}
\label{lem:zsp-irred-implies-AR-irred}
Let $S\subset\R^n$ be a semialgebraic set and let $E\subset S$ be an irreducible semialgebraic subset.
If $S$ has the zero-set property, then the $S$-$\AR$ closure $\Cl^\AR_S(E)$ of $E$ is $S$-$\AR$-irreducible.
In particular, $\Cl^\AR_M(E)$ is $\AR$-irreducible, for every Nash manifold $M$ containing $E$.
\end{lemma}

\begin{proof}
Let $X_1,\dots,X_r$ be the $S$-$\AR$-irreducible components of $\Cl^\AR_S(E)$. By assumption, there are functions $f_1,\dots,f_r\in\Aa(S)$ such that $\ZZ_S(f_i)=X_i$ for every $i$. Thus, $f_1\cdots f_r|_E=0$. By irreducibility of $E$, there is a $j$ such that $f_j|_E=0$, and hence $E\subset\ZZ_S(f_j)=X_j$.
Therefore, $\Cl^\AR_S(E)=X_j$\,.
\end{proof}

\begin{remark}
\label{rem:AR-irred-not-irred}
The converse of Lemma~\ref{lem:zsp-irred-implies-AR-irred} is not true. For instance, the set $S=\R\setminus\{0\}$ from Remark~\ref{rem:AR-not-induced} has $\AR$-irreducible $\AR$-closure in $\R$, namely $\R$ itself. But the ring $\Aa(S)$ is not a domain (consider the characteristic functions of intervals $(-\infty,0)$ and $(0,\infty)$). See below for less obvious examples, in which $S$ is connected.

On the other hand, if a semialgebraic set $S$ is $\AR$-closed in a Nash manifold $M$, then $S$ is irreducible in the sense of Definition~\ref{def:arcan-irreducible} if and only if it is $\AR$-irreducible in $M$. Indeed, this follows from the ring isomorphism in Theorem~\ref{thm:ASey-arc-an-extension} and the Arc-analytic Nullstellensatz (Theorem~\ref{thm:AR-Nullstellensatz}).
\end{remark}

\begin{example}[{\cite[Ex.\,4.6]{S}}]
\label{ex:1reducible}
Let $S$ be a subset of $M=\R^2$ defined as
\[
S=\{(x,y)\in\R^2:\, y^2=x(x-1)^2\;\wedge\;(y>0\;\vee\;0<x\leq1)\}\,.
\]
Then, $\Cl^\AR_M(S)$ is the cubic $\{y^2=x(x-1)^2\}$, which is $\AR$-irreducible in $M$. The set $S$ itself is not irreducible, because the function $y^2-x(x-1)^2$ has a factorization in $S$
\begin{equation}
\label{eq:cubic-branches}
y^2-x(x-1)^2\;=\;(y-\sqrt{x}(x-1))(y+\sqrt{x}(x-1))
\end{equation}
into functions arc-analytic on $S$.

On the other hand, $S$ as a subset of the (connected) Nash manifold $H=\{(x,y)\in\R^2: x>0\}$ has $\AR$-reducible $\AR$-closure. Indeed, the two $\AR$-irreducible components of $\Cl^\AR_H(S)$ are the zero loci $Z_1$, $Z_2$ of the factors from~\eqref{eq:cubic-branches}, which are arc-analytic (in fact, analytic) on the whole $H$ (see Fig.~1).
\end{example}

\begin{center}
\setlength{\unitlength}{.5cm}
\begin{picture}(29,10)
\makebox(9,10)[l]
{
\color{blue}
 \put(0.03,3.6){\circle{.23}}
 \put(4,4.5){\makebox(1,2){$S$}}
\begin{tikzpicture}
 \begin{axis}[axis lines=none, xmin = -0.5, xmax = 2, ymin = -1.5, ymax = 1.5]
  \addplot[white, domain=-1.5:-1, samples=100, variable=t]
	 ({t^2}, {t^3 - t});
  \addplot[blue, thick, domain=-1:-0.05, samples=100, variable=t]
	 ({t^2}, {t^3 - t});
  \addplot[blue, thick, domain=0.05:1.5, samples=100, variable=t]
	 ({t^2}, {t^3 - t});
 \end{axis}
 \draw[white, fill=black, fill opacity=0.1] (0.87,0) rectangle (5,4);
 \draw[black, dotted, opacity=0.75] (0.87,0) -- +(0,4);
 \draw (1.5,0.5) node [text=black, opacity=0.5] {H};
\end{tikzpicture}
}
\makebox(1,1)
{
}
\makebox(9,10)[c]
{
\color{black}
 \put(2.5,4.5){\makebox(1,2){$\Cl^\AR_M(S)$}}
\begin{tikzpicture}
 \begin{axis}[axis lines=none, xmin = -0.5, xmax = 2, ymin = -1.5, ymax = 1.5]
  \addplot[black, thick, domain=-1.5:1.5, samples=100, variable=t]
	 ({t^2}, {t^3 - t});
 \end{axis}
\end{tikzpicture}
}
\makebox(1,1)
{
}
\makebox(9,10)[r]
{
\color{blue}
 \put(0.03,3.6){\circle{.23}}
 \put(5.5,6){\makebox(1,2){$Z_1$}}
\color{red}
 \put(6,1){\makebox(1,2){$Z_2$}}
\color{black}
 \put(2.5,4.5){\makebox(1,2){$\Cl^\AR_H(S)$}}
\begin{tikzpicture}
 \begin{axis}[axis lines=none, xmin = -0.5, xmax = 2, ymin = -1.5, ymax = 1.5]
  \addplot[red, thick, domain=-1.5:-0.05, samples=100, variable=t]
	 ({t^2}, {t^3 - t});
  \addplot[blue, thick, domain=0.05:1.5, samples=100, variable=t]
	 ({t^2}, {t^3 - t});
 \end{axis}
 \draw[white, fill=black, fill opacity=0.1] (0.87,0) rectangle (5,4);
 \draw[black, dotted, opacity=0.75] (0.87,0) -- +(0,4);
 \draw (1.5,0.5) node [text=black, opacity=0.5] {H};
\end{tikzpicture}
}
\end{picture}
{~}\\
Figure 1
\end{center}
\medskip

\begin{example}
\label{ex:2reducible}
Let $B$ be the closed ball with radius $\sqrt{2}$ in $\R^3$ centered at $(0,1,-1)$,
\[
B=\{(x,y,z)\in\R^3:\, x^2+(y-1)^2+(z+1)^2\leq2\}\,,
\]
and let $S$ be the trace on $B$ of the Whitney umbrella,
\[
S=\{(x,y,z)\in B:\, x^2=zy^2\}\,.
\]
We claim that $S$ is not irreducible, but the $\AR$-closure $\Cl^\AR_U(S)$ is $\AR$-irreducible for an \emph{arbitrarily small} connected open semialgebraic neighbourhood $U$ of $S$ in $\R^3$ (and hence, in any Nash manifold containing $S$).
To see that $\Aa(S)$ is not a domain, consider functions $f,g:S\to\R$ defined as $f(x,y,z)=x^2+y^2$, and
\[
g(x,y,z)=\begin{cases}0 & \mathrm{if\ }z\geq0\\ z & \mathrm{if\ }z<0\,.\end{cases}
\]
Let $S^{(1)}$ and $S^{(2)}$ denote the one-dimensional and the two-dimensional locus of $S$, respectively (see Fig.~2\footnote{The 3D surface parametrization and LaTeX PGFPlots code for Figure~2 was generated with the assistance of Google's Gemini AI.}). Then, the arc-analytic function $f$ vanishes on $S^{(1)}$ and is non-zero on $S^{(2)}\setminus\{(0,0,0)\}$, while $g$ vanishes on $S^{(2)}$ and is non-zero on $S^{(1)}$. Also, $g$ is arc-analytic on $S$, since $g$ is analytic along every arc contained in either $S^{(1)}\cup\{(0,0,0)\}$ or in $S^{(2)}$, and there are no other analytic arcs in $S$.

On the other hand, if $U$ is an arbitrary open neighbourhood of $S$ in $\R^3$, then $U$ contains a small open ball centered at the origin, and hence the germ at the origin of the Whitney umbrella. It follows that $\Cl^\AR_U(S^{(2)})$ contains the part of the $z$-axis enclosed in $U$.
\end{example}

\begin{center}
\setlength{\unitlength}{.45cm}
\begin{picture}(15,15)
\makebox(15,15)[c]
{
%
% Generated by Google Gemini (start)
%
\begin{tikzpicture}
 \begin{axis}[
  axis lines=none,
  xmin=-1.2, xmax=1.2,
  ymin=0, ymax=2.2,
  zmin=-2.5, zmax=1.0,
  view={50}{20},
  colormap/redyellow,
  declare function={zmax(\y) = (sqrt(\y^4 + 8*\y + 4) - (\y^2 + 2)) / 2;},
  x post scale=1.6,
  y post scale=1.6,
  z post scale=3,
  ]
% 1. The Bounding Sphere
  \addplot3[
            surf,
            opacity=0.05,
            color=black,
            faceted color=black!75,
            samples=30,
            domain=0:360,     
            domain y=0:180,   
        ] (
            {sqrt(2)*sin(y)*cos(x)},        
            {1 + sqrt(2)*sin(y)*sin(x)},    
            {-1 + sqrt(2)*cos(y)}           
        );
  % 2. The Inner Constrained Surface
  \addplot3[
            surf,
            opacity=0.5,
            samples=30,
            domain=-1:1,       
            domain y=0.01:2.0  
        ] 
        (
            {x * y * sqrt(zmax(y))}, 
            {y},                     
            {x^2 * zmax(y)}          
        );
  % 3. The z-axis Piece Contained inside the Sphere (from z=-2 to z=0)
  \addplot3[
            blue,
            thick,
            samples y=0,
            domain=-2:0
        ] (0, 0, {x});
 \end{axis}
 \draw (5.5,3) node [text=black, opacity=0.5] {B};
\end{tikzpicture}
%
% Generated by Google Gemini (end)
%
\color{blue}
 \put(-11.07,10.3){\circle*{.23}}
 \put(-13.5,10){\makebox(1,1){$(0,0,0)$}}
 \put(-10.3,3){\makebox(1,1){$S^{(1)}$}}
\color{red}
 \put(-5,10){\makebox(1,1){$S^{(2)}$}}
}
\end{picture}
{~}\\
Figure 2
\end{center}
\medskip

The following result, generalizing Proposition~\ref{prop:ident-princ}, is due to Seyedinejad~\cite[Thm.\,3.6]{S}.

\begin{proposition}[\textbf{Arc-Analytic Identity Principle}]
\label{prop:general-ident-princ}
Let $S$ be a semialgebraic set in a Nash manifold $M$. Then, $S$ is irreducible if and only if, for any function $f\in\Aa(S)$, if $f|_U=0$ for a subset $U\subset S$ with $\dim{U}=\dim{S}$ then $f\equiv0$.
\end{proposition}

\begin{proof}
Let $k=\dim{S}$.
Suppose that $S$ is irreducible, and a function $f\in\Aa(S)$ vanishes on a subset $U\subset S$ with $\dim{U}=\dim{S}$. Let us denote the entire $\ZZ_S(f)$ by $U$ again. Put $V=S\setminus U$, and $W=\Cl_S(V)\cap U$. Note that $\dim{W}<k$, hence also $\dim\Cl^\AR_M(W)<k$, by Proposition~\ref{prop:alg-closure-dim}. By Theorem~\ref{thm:ASey-zero-set-thm}, there exists a function $h\in\Aa(M)$ such that $\Cl^\AR_M(W)=\ZZ_M(h)$. Consider functions $g_1,g_2:S\to\R$ defined as follows
\[
g_1(x)=\begin{cases}h(x) &\mathrm{if\ }x\in U\\ 0 &\mathrm{if\ }x\in V\end{cases}\qquad\mathrm{and}\qquad
g_2(x)=\begin{cases}h(x) &\mathrm{if\ }x\in V\\ 0 &\mathrm{if\ }x\in U\end{cases}\,.
\]
We claim that $g_1$ and $g_2$ are arc-analytic on $S$. Note also that $g_1$ is not identically zero on $U$, since $\dim\ZZ_M(h)<\dim{U}$. On the other hand, $g_1g_2\equiv0$, so $g_2\equiv0$ by irreducibility of $S$. It follows that $V\subset\ZZ_M(h)$, and thus $fh=0$ on $S$. Therefore, $f\equiv0$ by irreducibility of $S$ again.

It remains to verify that $g_1,g_2\in\Aa(S)$. By arc-symmetry of $U$ in $S$, it suffices to show that the $g_j$ are analytic along every arc $\gamma:(-1,1)\to S$ whose image is either entirely contained in $U$, or else intersects $U$ only at isolated points and is thus contained in $V\cup W$. For any such $\gamma$ of the first kind, we have $g_1\circ\gamma=h\circ\gamma$ and $g_2\circ\gamma\equiv0$. And for every $\gamma$ of the second kind, $g_1\circ\gamma\equiv0$ and $g_2\circ\gamma=h\circ\gamma$. This completes the proof of the forward implication.

For the proof of the converse, let $f,g\in\Aa(S)$ be such that $fg\equiv0$. Then, one of the sets $\ZZ_S(f)$ and $\ZZ_S(g)$ contains a $k$-dimensional subset of $S$ (by Lemma~\ref{lem:dim-drop} applied to a connected component of $\Reg_k{S}$), and hence the whole of $S$, by assumption. Therefore, $f\equiv0$ or else $g\equiv0$, showing that $\Aa(S)$ is a domain.
\end{proof}

\begin{corollary}[{\cite[Cor.\,3.7]{S}}]
\label{cor:smooth-connected-is-irred}
Every connected Nash manifold is an irreducible semialgebraic set.
\end{corollary}

\begin{proof}
Let $M$ be a connected Nash manifold, and let $f\in\Aa(M)$ be such that $f|_U=0$ for some semialgebraic set $U\subset M$ of dimension $\dim{U}=\dim{M}$. By Proposition~\ref{prop:general-ident-princ}, it suffices to show that $f\equiv0$. Since $\ZZ_M(f)$ is an arc-symmetric subset of $M$, the claim follows from Lemma~\ref{lem:dim-drop}.
\end{proof}

Another important consequence of Proposition~\ref{prop:general-ident-princ} is the following corollary, which we will use to establish the equality between the geometric dimension of a semialgebraic set $S$ and the Krull dimension of its ring $\Aa(S)$ of arc-analytic functions (see Theorem~\ref{thm:dim-equals-Krull-dim} below).

\begin{corollary}
\label{cor:irred-implies-dim-drop}
Let $S\subset\R^n$ be a semialgebraic set, which has the zero-set property. Suppose $X,Y\subset S$ are $S$-$\AR$-closed sets, $Y\subset X$, and $X$ is irreducible as a semialgebraic set. Then,
\[
\dim{Y}=\dim{X}\ \Longrightarrow\ Y=X\,.
\]
\end{corollary}

\begin{proof}
Suppose $\dim{Y}=\dim{X}$.
By assumption, there is a function $f\in\Aa(S)$ such that $\ZZ_S(f)=Y$. Then, the restriction $f|_X$ is an arc-analytic function on $X$, which vanishes on its semialgebraic subset of dimension $\dim{X}$. It follows from Proposition~\ref{prop:general-ident-princ} that $f|_X=0$, and so $X\subset\ZZ_S(f)=Y$.
\end{proof}
\medskip

We conclude this section with a characterization of irreducibility of pure-dimensional semialgebraic sets in terms of blow-ups, {\`a} la Theorem~\ref{thm:AR-smoothing} and Corollary~\ref{cor:1-AR-irred}.

\begin{theorem}
\label{thm:irred-criterion}
Let $S$ be a closed semialgebraic set in a Nash manifold $M$, of pure dimension $k$, and let $C_1,\dots,C_s$ be all the connected components of $\Reg_k{S}$.
The following conditions are equivalent:
\begin{itemize}
\item[(i)] The ring $\Aa(S)$ is an integral domain.
\item[(ii)] For every $j$, there exists a nonempty open subset $U_j$ of $C_j$, such that if \;$i,j\in\{1,\dots,s\}$ and $p\in U_i$, $q\in U_j$ are arbitrary then there exist indices $j_1,\dots,j_t\in\{1,\dots,s\}$, with $j_1=i$ and $j_t=j$, points $p^0_{j_l},p^1_{j_l}\in U_{j_l}$, with $p=p^0_{j_1}$ and $q=p^1_{j_t}$, and analytic arcs $\gamma_{j_l}:[0,1]\to\Cl_S(\Reg_k{S})$, all such that $\gamma_{j_l}(0)=p^1_{j_l}$ and $\gamma_{j_l}(1)=p^0_{j_{l+1}}$ for $l=1,\dots,t-1$.
\item[(iii)] For every finite composition of blow-ups
\[
\pi=\pi_1\circ\dots\circ\pi_s:\wt{M}=M_s\to M_{s-1}\to\dots\to M_1\to M_0=M
\]
with smooth Nash centres such that the exceptional locus $E$ of $\pi$ is nowhere dense in $S$, there exists a unique connected component $\wt{S}$ of $S'=\Cl_{\wt{M}}(\pi^{-1}(S\setminus E))$ such that $\pi(\wt{S})=S$.
\end{itemize}
Moreover, the implication (i)\,$\Rightarrow\,$(iii) \;holds without the pure-dimensionality assumption on $S$.
\end{theorem}

\begin{proof}
(ii)\,$\Rightarrow\,$(i):
Let $f,g\in\Aa(S)$ be such that $fg\equiv0$. At least one of the semialgebraic sets $\ZZ_S(f)$ and $\ZZ_S(g)$ must be of dimension $k$ (by Lemma~\ref{lem:dim-drop} applied to a connected component of $\Reg_k{S}$). Say, $\dim\ZZ_S(f)=k$. Since $\ZZ_S(f)$ is an arc-symmetric subset of $S$, there exists a connected component $C_i$ of $\Reg_k{S}$ such that $f|_{C_i}=0$. By (ii), it follows that $f|_{C_j}=0$ for every connected component $C_j$ of $\Reg_k{S}$. Then, $f\equiv0$, by continuity of $f$ and since $S=\Cl_S(\Reg_k{S})$.
\smallskip

(i)\,$\Rightarrow\,$(iii): The set $S$ being closed implies that $\pi^{-1}(S)$ is closed in $\wt{M}$ and
\[
S'=\Cl_{\wt{M}}(\pi^{-1}(S\setminus Y))=\Cl_{\pi^{-1}(S)}(\pi^{-1}(S\setminus Y))\,.
\]
Therefore, $S'\subset\pi^{-1}(S)$ and $\pi(\wt{S})\subset S$ for every component $\wt{S}$ of $S'$.

On the other hand, $S\subset\pi(S')$.
Indeed, since the exceptional locus $E$ of $\pi$ is nowhere dense in $S$, we have $\Cl_S(S\setminus E)=S$.
Therefore, as $\pi(S')$ is closed as the image of a closed set by a proper map, we get
\begin{equation}
\label{eq:1}
S=\Cl_S(S\setminus E)=\Cl_M(\pi(\pi^{-1}(S\setminus E)))\cap S=\pi(\Cl_{\wt{M}}(\pi^{-1}(S\setminus E)))\cap S\subset\pi(S')\,.
\end{equation}

Next, let $C$ be a connected component of $S\setminus E$ of dimension $k$. The set $\pi^{-1}(C)$ is connected (since $\pi$ is a homeomorphism over $M\setminus E$), and hence contained in a single connected component $\wt{S}_C$ of $S'$. We claim that then $S\setminus E\subset\pi(\wt{S}_C)$.

Suppose otherwise. Set $\wt{T}=S'\setminus\wt{S}_C$. By~\eqref{eq:1}, $S\setminus C\subset\pi(\wt{T})$. Define a function $g:S'\to\R$ as
\[
g|_{\wt{S}_C}=0\qquad\mathrm{and}\qquad g|_{\wt{T}}=h|_{\wt{T}}\,,
\]
where $h\in\Aa(\wt{M})$ is such that $h^{-1}(0)=\pi^{-1}(E)$. Then, $g\in\Aa(S')$ and $g|_{S'\cap\pi^{-1}(E)}=0$. Consequently, the function $f:S\to\R$ defined as
\[
f(x)=\begin{cases}g\circ\pi^{-1}(x) &\mathrm{if\ }x\in S\setminus E\\ 0 &\mathrm{if\ }x\in S\cap E\end{cases}
\]
is in $\Aa(S)$. To see this, let $\gamma:(-1,1)\to S$ be an arbitrary analytic arc, and let $\wt{\gamma}:(-1,1)\to S'$ be such that $\gamma=\pi\circ\wt{\gamma}$. Then,
\[
f\circ\gamma(t)=g\circ\wt\gamma(t)\qquad\mathrm{for\ all\ }t\in(-1,1)\,.
\]
Indeed, if $\gamma(t)\in E$, then $f(\gamma(t))=0$ by definition, and $\wt\gamma(t)\in\pi^{-1}(E)$, whence $h(\wt\gamma(t))=0$ and so $g(\wt\gamma(t))=0$ as well.
If, in turn, $\gamma(t)\in S\setminus E$, then $f\circ\gamma(t)=g\circ\pi^{-1}\circ\pi\circ\wt\gamma(t)=g\circ\wt\gamma(t)$.

Note that $f|_C=0$, but $f$ is not identically zero on $S$, because $f(x)\neq0$ for $x\in S\setminus(E\cup\pi(\wt{S}_C))$. This contradicts $\Aa(S)$ being a domain, by Proposition~\ref{prop:general-ident-princ}. Thus, $S\setminus E\subset\pi(\wt{S}_C)$.

Note, finally, that $\pi|_{S'}:S'\to S$ is a proper map, and $\wt{S}_C$ is a closed subset of $S'$ (as its connected component). Therefore, $\Cl_S(\pi(\wt{S}_C))=\pi(\wt{S}_C)$, and hence
\[
\pi(\wt{S}_C)=\Cl_S(\pi(\wt{S}_C))\supset\Cl_S(S\setminus E)=S\,,
\]
which completes the proof of this implication.
\smallskip

(iii)\,$\Rightarrow\,$(ii):
Let $X=\Cl^\Na_M(S)$ be the Nash closure of $S$ in $M$, and let $\pi:\wt{M}\to M$ be an embedded desingularization of $X$. By Theorem~\ref{thm:AR-smoothing}, there are connected components $\wt{E}_1,\dots,\wt{E}_t$ of the strict transform $\wt{X}$ of $X$ by $\pi$, such that $\pi(\wt{E}_1\cup\dots\cup\wt{E}_t)\supset\Cl(\Reg_k{X})$, and hence $\pi(\wt{E}_1\cup\dots\cup\wt{E}_t)\supset S$. Let $Z$ denote the exceptional locus of $\pi$. By assumption, a unique connected component $\wt{S}$ of $\Cl_{\pi^{-1}(S)}(\pi^{-1}(S\setminus Z))$ maps onto $S$. Being a connected set, $\wt{S}$ is contained in precisely one of the $\wt{E}_1,\dots,\wt{E}_t$, which we shall henceforth denote by $\wt{E}$.

Let $Y$ be the Nash closure (in $\wt{E}$, and hence in $\wt{M}$) of the semialgebraic set $\wt{S}\setminus\Reg_k{\wt{S}}$. Then, $Y$ is of dimension strictly less than $k$ and it contains the frontiers of all the connected components of $\Reg_k{\wt{S}}$.
Let $h\in\Na(\wt{E})$ be such that $h^{-1}(0)=Y$.
By resolution of singularities, there is a finite composition of blow-ups $\sigma:\widehat{E}\to\wt{E}$ with smooth Nash centres, such that $h\circ\sigma$ has only normal crossings. (Since $\wt{E}$ is a closed Nash submanifold of $\wt{M}$, we can regard $\sigma$ as a composite of blow-ups of $\wt{M}$; i.e., $\sigma:\widehat{M}\to\wt{M}$, where $\widehat{E}$ is the strict transform of $\wt{E}$ in $\widehat{M}$.)

Let $\Sg$ denote the exceptional divisor of $\pi\circ\sigma$.
By assumption, there exists a unique connected component $\wh{S}$ of $\Cl_{(\pi\circ\sigma)^{-1}(S)}((\pi\circ\sigma)^{-1}(S)\setminus\Sg)$ such that $\pi\circ\sigma(\wh{S})=S$. Let $S_1,\dots,S_r$ be all the connected components of $\Reg_k(\wh{S}\setminus\Sg)$. By construction, $\Cl_{\wh{S}}(\Reg_k(\wh{S}\setminus\Sg))=\wh{S}$, and locally at every point $p\in\wh{S}\setminus\Reg_k(\wh{S}\setminus\Sg)$ there is a coordinate chart $U^p$ such that, for every $j=1,\dots,r$, either $S_j\cap U^p=\varnothing$ or else $S_j\cap U^p$ contains a union of open quadrants $\{\lambda_i(x_i)>0:i=1,\dots,k\}$, where $\lambda_i(x)=x$ for all $x\in\R$ or $\lambda_i(x)=-x$ for all $x\in\R$, and $(x_1,\dots,x_k)$ are local coordinates in $U^p$ centered at $p$.

Fix a point $p\in\wh{S}\setminus\Reg_k(\wh{S}\setminus\Sg)$ for which $U^p$ intersects more than one of the $S_1,\dots,S_r$. Let $S_{i_1}$ and $S_{i_2}$ be any two such components, and let $Q_1$ and $Q_2$ be open quadrants (at the origin in $\R^k$) such that $Q_1\cap U^p\subset S_{i_1}$ and $Q_2\cap U^p\subset S_{i_2}$.
Since any point of $Q_1$ can be joined with any point of $Q_2$ by an analytic arc passing through the origin, it follows that
\begin{multline}
\label{eq:2}
\mathrm{there\ exist\ open\ } T_{i_1}\subset S_{i_1},T_{i_2}\subset S_{i_2}\mathrm{\ such\ that\ for\ all\ }q_1\in T_{i_1},q_2\in T_{i_2}\\
\mathrm{there\ is\ an\ analytic\ }\widehat\gamma:[0,1]\to\widehat{S}\mathrm{\ with\ \,}\widehat\gamma(0)=q_1,\widehat\gamma(1)=q_2.
\end{multline}
Note finally that, since $\pi\circ\sigma|_{\widehat{M}\setminus\Sg}$ is a homeomorphism, the index set $I=\{1,\dots,r\}$ partitions into $I=I_1\sqcup\dots\sqcup I_s$, where $I_j=\{i\in I:\pi\circ\sigma(S_i)\subset C_j\}$. We have
\[
C_j=(\bigcup_{i\in I_j}\pi\circ\sigma(\Cl(S_i)))\cap C_j,\qquad\mathrm{for\ all\ }j=1,\dots,s\,.
\]
Moreover, by~\eqref{eq:2},
\begin{multline}
\label{eq:3}
\mathrm{for\ all\ }p_1\in\pi\circ\sigma(T_{i_1})\mathrm{\,\ and\ \,}p_2\in\pi\circ\sigma(T_{i_2})\,,\\
\mathrm{there\ is\ an\ analytic\ }\gamma:[0,1]\to S\mathrm{\ with\ \,}\gamma(0)=p_1,\gamma(1)=p_2.
\end{multline}

To complete the proof, suppose that (ii) does not hold. Then, by~\eqref{eq:3}, there exists a proper index subset $J\subset\{1,\dots,s\}$ such that, for all $j_1\in J$ and $j_2\in\{1,\dots,s\}\setminus J$, we have
\[
\Cl(S_{i_1})\cap\Cl(S_{i_2})\cap\widehat{S}=\varnothing,\quad\mathrm{for\ all\ }i_1\in I_{j_1}\mathrm{\ and\ } i_2\in I_{j_2}\,.
\]
It follows that $\widehat{S}$ is a disjoint union of $\widehat{S}_1$ and $\widehat{S}_2$, where
\[
\widehat{S}_1=(\bigcup_{j\in J}\bigcup_{i\in I_j}\Cl(S_i))\cap\widehat{S}\quad\mathrm{and}\quad
\widehat{S}_2=(\bigcup_{j\notin J}\bigcup_{i\in I_j}\Cl(S_i))\cap\widehat{S}\,.
\]
Since $\dim(\pi\circ\sigma(\widehat{S}_1))\cap C_j<k$ for all $j\notin J$ and $\dim(\pi\circ\sigma(\widehat{S}_2))\cap C_j<k$ for all $j\in J$, then neither $\widehat{S}_1$ nor $\widehat{S}_2$ is mapped onto $S$, which contradicts condition (iii).
\end{proof}

\begin{corollary}
\label{cor:irreducible-dim-drop}
Let $S\subset\R^n$ be an irreducible semialgebraic set of pure dimension. Let $X\subset S$ be $S$-$\AR$-closed. Then, $\dim{X}=\dim{S}$ implies $X=S$.
\end{corollary}

\begin{proof}
Let $C_1,\dots,C_s$ be the connected components of $\Reg{S}$. Since $\dim{X}=\dim{S}$, then $X$ contains a nonempty open subset of some $C_i$, and so $C_i\subset X$, by Lemma~\ref{lem:dim-drop}. It follows, by Theorem~\ref{thm:irred-criterion} and arc-symmetry of $X$, that $X$ contains a nonempty open subset of every $C_j$, $j=1,\dots,s$. Hence, $\Reg{S}=\bigcup_jC_j\subset X$, by Lemma~\ref{lem:dim-drop} again, and $S=\Cl_S(\Reg{S})\subset X$, since $X$ is closed in $S$.
\end{proof}

\vspace{2ex}
%%%%%%%%%%%%%%%%%%%%%%%%%%%%%%%%%%%%%%%%%%%%%%%%%%
%%%%%%%%%%%%%%%%% Section %%%%%%%%%%%%%%%%%%%%%%%%
%%%%%%%%%%%%%%%%%%%%%%%%%%%%%%%%%%%%%%%%%%%%%%%%%%

\section{Irreducible components of a semialgebraic set}
\label{sec:irred-comps}

\begin{lemma}
\label{lem:Hadi-4.1-general}
Let $S\subset\R^n$ be a semialgebraic set. Let $E_1,\dots,E_r\subset S$ be irreducible semialgebraic sets such that $S\subset\Cl(E_1\cup\cdots\cup E_r)$.
Then, $\bigcap_{i=1}^r\II_S(E_i)=(0)$.
Moreover, the set of minimal prime ideals in $\Aa(S)$ coincides with the set of minimal elements of $\{\II_S(E_1),\dots,\II_S(E_r)\}$.
\end{lemma}

\begin{proof}
As $S\subset\Cl(E_1\cup\cdots\cup E_r)$, the continuity of arc-analytic functions implies $\bigcap_{i=1}^r\II_S(E_i)=(0)$.
By Proposition~\ref{prop:prime-when-irred}, each $\II_S(E_i)$ is a prime ideal in $\Aa(S)$.
It remains to show that all the minimal primes in $\Aa(S)$ are contained in the family $\{\II_S(E_1),\dots,\II_S(E_r)\}$. To this end, let $\pp$ be a minimal prime in $\Aa(S)$. Because $\bigcap_{i=1}^r\II_S(E_i)\subset\pp$ and $\pp$ is prime, $\II_S(E_i)\subset\pp$ for some $i$. Thus, $\pp=\II_S(E_i)$ by minimality of $\pp$.
\end{proof}

\begin{corollary}[{\cite[Lem.\,4.1]{S}}]
\label{cor:Hadi-4.1}
If $C_1,\dots,C_r$ are the connected components of the regular locus $\Reg{S}$, then $\bigcap_{i=1}^r\II_S(C_i)=(0)$.
Moreover, the minimal elements of $\{\II_S(C_1),\dots,\II_S(C_r)\}$ coincide with the minimal primes in $\Aa(S)$.
In particular, there are only finitely many minimal primes in $\Aa(S)$.
The same is true if the $C_i$ are replaced by the $S$-$\AR$-irreducible components of $S$.
\end{corollary}

\begin{proof}
Both statements follow immediately from Lemma~\ref{lem:Hadi-4.1-general}.
For the first, notice that $S=\Cl_S(\Reg{S})$ and that each $C_i$ is irreducible by Corollary~\ref{cor:smooth-connected-is-irred}.
For the second, recall that $S$ is the union of its $S$-$\AR$-irreducible components and that by Proposition~\ref{prop:s-AR-irred-implies-arcan-irred} each $S$-$\AR$-irreducible component is an irreducible set.
\end{proof}

Following Seyedinejad~\cite[Def.\,4.2]{S}, we define irreducible components of a semialgebraic set as follows.

\begin{definition}
\label{def:irred-comps}
Given a semialgebraic set $S$, let $\pp_1,\dots,\pp_r$ be the minimal prime ideals in \,$\Aa(S)$. Put $S_i=\ZZ_S(\pp_i)$ for every $i$. The $\{S_1,\dots,S_r\}$ are called the \emph{irreducible components} of $S$.
\end{definition}

\begin{remark}
\label{rem:if-AR-closed-same-irred}
In the case that $S$ is an $\AR$-closed set in a Nash manifold, the family of irreducible components of $S$ as defined above coincides with the family of its $\AR$-irreducible components. This follows from~\cite[Prop.\,2]{ASe1}, which states that in an $\AR$-closed set $S$ the $\AR$-irreducible components are precisely the zero-sets of the minimal primes in $\Aa(S)$.
\end{remark}

We establish now some basic properties of irreducible components.

\begin{proposition}
\label{prop:irred-comps-no-inclusions}
Let $S$ be a semialgebraic set and let $\pp_1,\dots,\pp_r$ be all the (distinct) minimal prime ideals of $\Aa(S)$. Let $S_1,\dots,S_r$ be the corresponding irreducible components of $S$. Then, for all $i,j\in\{1,\dots,r\}$, $i\neq j$ implies $S_i\not\subset S_j$.
\end{proposition}

\begin{proof}
Suppose $S_i\subset S_j$ for some $i,j\in\{1,\dots,r\}$. Then,
\[
\II_S(S_i)\supset\II_S(S_j)=\II_S(\ZZ_S(\pp_j))\supset\pp_j\,.
\]
On the other hand, let $C_i$ be a connected component of $\Reg{S}$, for which $\pp_i=\II_S(C_i)$. Then, $C_i\subset S_i$, and hence $\II_S(S_i)\subset\II_S(C_i)=\pp_i$.
It follows that $\pp_i\supset\pp_j$, whence $i=j$, by minimality of $\pp_i$.
\end{proof}

\begin{proposition}
\label{prop:irred-iff-one-irred-comp}
A semialgebraic set $S$ is irreducible if and only if $S$ has only one irreducible component.
\end{proposition}

\begin{proof}
By Proposition~\ref{prop:irred-comps-no-inclusions}, $S$ has only one irreducible component precisely when there is only one minimal prime ideal $\pp$ in $\Aa(S)$. Since the zero ideal $(0)$ is equal to the intersection of all the minimal primes, it follows that $S$ has only one irreducible component if and only if $(0)$ is prime, that is, $\Aa(S)$ is a domain.
\end{proof}

\begin{proposition}[cf.~{\cite[Prop.\,4.3]{S}}]
\label{prop:irred-comps-basics}
Let $S\subset\R^n$ be a semialgebraic set and let $S_1,\dots,S_r$ be its irreducible components. We have the following:
\begin{itemize}
\item[(i)] $S=S_1\cup\cdots\cup S_r$
\item[(ii)] If $Z$ is an irreducible semialgebraic subset of $S$, then $Z\subset S_i$ for some $i$
\item[(iii)] If $Z$ is an irreducible semialgebraic subset of $S$ such that $Z\supset S_i$ for some $i$, then $Z=S_i$.
\end{itemize}
Moreover, if $S$ has the zero-set property, we have also:
\begin{itemize}
\item[(iv)] The irreducible components and $S$-$\AR$-irreducible components of $S$ coincide
\item[(v)] Each $S_i$ is an irreducible semialgebraic set.
\end{itemize}
\end{proposition}

\begin{proof}
Let $\pp_1,\dots,\pp_r$ be the underlying minimal primes of $S_1,\dots,S_r$ respectively. Then,
\[
S=\ZZ_S(0)=\ZZ_S(\bigcap_{i=1}^r\pp_i)=\bigcup_{i=1}^r\ZZ_S(\pp_i)=\bigcup_{i=1}^r S_i\,,
\]
which proves (i).

For (ii), by Proposition~\ref{prop:prime-when-irred}, the vanishing ideal $\II_S(Z)$ is prime and hence $\pp_i\subset\II_S(Z)$ for some~$i$. Thus,
\[
Z \subset\ZZ_S(\II_S(Z))\subset\ZZ_S(\pp_i)=S_i\,.
\]
Next, let $Z$ and $S_i$ be as stated in (iii). By Corollary~\ref{cor:Hadi-4.1}, there is a connected component $C_i$ of $\Reg{S}$ such that $\II_S(C_i)=\pp_i$. Hence,
\[
\II_S(\ZZ_S(\pp_i))=\II_S(\ZZ_S(\II_S(C_i)))=\II_S(C_i)=\pp_i\,,
\]
and so $\II_S(S_i)=\pp_i$.
Therefore, $Z\supset S_i$ implies $\II_S(Z)\subset\pp_i$. But $\II_S(Z)$ is prime, so $\II_S(Z)=\pp_i$ by minimality of $\pp_i$.
This, in turn, implies that $Z\subset\ZZ_S(\pp_i)=S_i$, and thus $Z=S_i$, proving (iii).

For (iv), let $X_1,\dots,X_t$ be the $S$-$\AR$-irreducible components of $S$. By Corollary~\ref{cor:Hadi-4.1}, it suffices to show that $\II_S(X_i)\subset\II_S(X_j)$ implies $X_i=X_j$. Indeed, using the zero-set property, we find a function $f_i\in\Aa(S)$ such that $\ZZ_S(f_i)=X_i$. Then, $f_i\in\II_S(X_i)\subset\II_S(X_j)$, so $X_j\subset\ZZ_S(\II_S(X_j))\subset\ZZ_S(f_i)=X_i$. Hence, $X_i=X_j$ by minimality of the decomposition into $S$-$\AR$-irreducible components. This proves~(iv).
Lastly, by Proposition~\ref{prop:s-AR-irred-implies-arcan-irred}, the $S$-$\AR$-irreducible components of $S$ are irreducible, proving (v).
\end{proof}

\begin{question}
\label{question:are-irred-comps-irred}
Are the irreducible components of every semialgebraic set (not having the ZSP) necessarily irreducible in the sense of Definition~\ref{def:arcan-irreducible}?
\end{question}

\begin{corollary}
\label{cor:irred-comps-are-max-irreds}
Let $S\subset\R^n$ be a semialgebraic set satisfying the zero-set property. Then, the maximal irreducible subsets of $S$ (with respect to inclusion) exist and are precisely the irreducible components of $S$.
\end{corollary}

\begin{proof}
This follows immediately from Proposition~\ref{prop:irred-comps-basics}(ii) and (v).
\end{proof}

\begin{question}
\label{question:are-there-always-max-irreds}
Does every semialgebraic set (not having the ZSP) necessarily have maximal irreducible subsets?
\end{question}

\begin{corollary}
\label{cor:ZSP-implies-irred-iff-AR-irred}
Let $S\subset\R^n$ be a semialgebraic set satisfying the zero-set property. Then, $S$ is $S$-$\AR$-irreducible if and only if $S$ is an irreducible semialgebraic set.
\end{corollary}

\begin{proof}
The forward implication is always true, by Proposition~\ref{prop:s-AR-irred-implies-arcan-irred}.
Conversely, suppose that $S$ is not $S$-$\AR$-irreducible. Then, $S$ has at least two distinct $S$-$\AR$-irreducible components, $S_1\neq S_2$. By Proposition~\ref{prop:irred-comps-basics}(iv), $S_1$, $S_2$ are among the irreducible components of $S$. Therefore, $S$ is not irreducible, by Proposition~\ref{prop:irred-iff-one-irred-comp}.
\end{proof}

As an application of the above results, let us now work out the decomposition of a semialgebraic set in two examples.

\begin{example}[{\cite[Ex.\,4.6]{S}}]
\label{ex:nodal-cubic-irred-decomp}
Let $S$ be the semialgebraic subset of $\R^2$ from Example~\ref{ex:1reducible}, defined by the formula
\[
y^2 = x(x-1)^2 \;\land\; (y > 0 \;\lor\; 0 < x \leq 1)\,.
\]
As a set of dimension one, $S$ has the arc-analytic extension property (see Corollary~\ref{cor:1-dim-has-AEP} below), and hence $S$ also has the zero-set property (Proposition~\ref{prop:AEP-implies-ZSP} below).

Taking a clue from the factorization
\[
y^2-x(x-1)^2 = (y-\sqrt{x}(x-1))(y+\sqrt{x}(x-1))\,,
\]
let $S_1$ and $S_2$ be the zero loci in $S$ of $y-\sqrt{x}(x-1)$ and $y+\sqrt{x}(x-1)$, respectively.
As traces of analytic sets from the right half-plane $H=\{(x,y)\in\R^2:x>0\}$, both $S_1$ and $S_2$ are arc-symmetric subsets of $S$, and clearly, both are $S$-$\AR$-irreducible. Thus, as $S=S_1\cup S_2$ and neither $S_1$ nor $S_2$ is contained in the other, it follows that $\{S_1, S_2\}$ are the $S$-$\AR$-irreducible components of $S$. Hence, by Proposition~\ref{prop:irred-comps-basics}(iv), $S_1\cup S_2$ is the decomposition of $S$ into irreducible components.
\end{example}

\begin{example}
\label{ex: Whitney-in-a-ball-irred-decomp}
Let $S$ be the semialgebraic subset of $\R^3$ from Example~\ref{ex:2reducible}, defined by the formula
\[
x^2=zy^2 \;\land\; x^2+(y-1)^2+(z+1)^2\leq2\,,
\]
so that $S$ is the trace of the Whitney umbrella in the closed ball with radius $\sqrt{2}$ centered at $(0,1,-1)$ in $\R^3$.
Let $S^{(1)}$ and $S^{(2)}$ denote the one-dimensional and the two-dimensional locus of $S$, respectively. We claim that $S_1=S^{(1)}\cup\{(0,0,0)\}$ and $S_2=S^{(2)}$ are the irreducible components of $S$.

To prove the claim, consider functions $f,g:S\to\R$ defined as $f(x,y,z)=x^2+y^2$, and
\[
g(x,y,z)=\begin{cases}0 & \mathrm{if\ }z\geq0\\ z & \mathrm{if\ }z<0\,.\end{cases}
\]
We have $f,g\in\Aa(S)$. The case for $f$ is obvious, and $g$ is arc-analytic since it is analytic along every arc contained in either $S^{(1)}\cup\{(0,0,0)\}$ or in $S^{(2)}$, and there are no other analytic arcs in $S$. Further, we have $\ZZ_S(f)=S_1$ and $\ZZ_S(g)=S_2$.

Consider now the sets $S'_1=S^{(1)}$ and $S'_2=S^{(2)}\cap\{x^2+(y-1)^2+(z+1)^2<2\}$. 
As connected smooth semialgebraic sets, $S'_1$ and $S'_2$ are irreducible (Corollary~\ref{cor:smooth-connected-is-irred}). Therefore, $S_1=\Cl_S(S'_1)$ and $S_2=\Cl_S(S'_2)$ are irreducible, by Lemma~\ref{lem:irred-implies-closure-irred} below.
As the zero-loci of $f$ and $g$, respectively, $S_1$ and $S_2$ are maximal irreducible subsets of $S$. Indeed, for if $S_i\subsetneq T$ ($i=1$ or $2$) for some semialgebraic set $T\subset S$, then $f|_T\!\cdot g|_T=0$ while $f|_T\neq0$ and $g|_T\neq0$, meaning that $T$ is not irreducible.

Lastly, note that, as a basic closed semialgebraic set, $S$ has the arc-analytic extension property (see Theorem~\ref{thm:basic-closed-has-AEP} below), and hence also the zero-set property (Proposition~\ref{prop:AEP-implies-ZSP} below). It thus follows by Corollary~\ref{cor:irred-comps-are-max-irreds} that $S_1$ and $S_2$ are the irreducible components of $S$.
\end{example}

\begin{lemma}
\label{lem:irred-implies-closure-irred}
Let $S\subset\R^n$ be a semialgebraic set, and let $T$ be an arbitrary semialgebraic set satisfying $S\subset T\subset\Cl(S)$.
Suppose that $S$ is irreducible. Then, $T$ is irreducible as well.
\end{lemma}

\begin{proof}
Let $T$ be as above, and let $f_1,f_2\in\Aa(T)$ be such that $f_1f_2=0$. Then, $f_1|_S\cdot f_2|_S=0$, hence by irreducibility of $S$ one of the functions vanishes on $S$; say, $f_1|_S=0$. Since $f_1$ is continuous on $T$, its zero-set is a closed subset of $T$. Thus,
\[
\ZZ_T(f_1)=\Cl(\ZZ_T(f_1))\cap T\supset\Cl(S)\cap T=T\,,
\]
and so $f_1=0$, proving that $\Aa(T)$ is a domain.
\end{proof}

\vspace{2ex}
%%%%%%%%%%%%%%%%%%%%%%%%%%%%%%%%%%%%%%%%%%%%%%%%%%
%%%%%%%%%%%%%%%%% Section %%%%%%%%%%%%%%%%%%%%%%%%
%%%%%%%%%%%%%%%%%%%%%%%%%%%%%%%%%%%%%%%%%%%%%%%%%%

\section{Generalized arc-analytic Nullstellensatz}
\label{sec:arcan-null}

Kurdyka's weak Nullstellensatz (\cite[Prop.\,6.5]{Kur}) generalizes easily to our setting. We include the proof for the reader's convenience.

\begin{proposition}
\label{prop:weak-nullstellensatz}
Let $S$ be a locally closed semialgebraic subset of $\R^n$, and suppose $f,g\in\Aa(S)$ are such that $\ZZ_S(f)\subset\ZZ_S(g)$.
Then, there exist $h\in\Aa(S)$ and a positive integer $k$ such that $g^k=fh$.
\end{proposition}

\begin{proof}
Since $f,g$ are continuous semialgebraic functions on a locally closed set, it follows from \cite[Thm.\,2.6.6]{BCR} that there are a continuous semialgebraic function $h_1:S\to\R$ and a positive integer $k_1$ such that $g^{k_1}=fh_1$. Define $h=gh_1$, and $k=k_1+1$. Then $g^k=fh$, and we claim that $h$ is arc-analytic.

To prove the claim, let $\gamma:(-1,1)\to S$ be an analytic arc. Since $g\circ\gamma$ is analytic, it follows from the (analytic) Identity Principle that $g\circ\gamma$ is identically zero, or else its zeroes are isolated points in $(-1,1)$. In the first case, we have $h\circ\gamma=(g\circ\gamma)\cdot(h_1\circ\gamma)=0$ and we're done. For the second case, we have equality
\[
(h_1\circ\gamma)(t)=\frac{(g^{k_1}\circ\gamma)(t)}{(f\circ\gamma)(t)}
\]
for all $t$ at which $g\circ\gamma$ (and hence $f\circ\gamma$) does not vanish. Since $h_1\circ\gamma$ is continuous on the whole $(-1,1)$, the singularities of the right hand side meromorphic function are removable. Therefore, $h_1\circ\gamma$ and hence $h\circ\gamma$ is analytic.
\end{proof}

\begin{remark}
\label{rem:locally-closed-is-necessary}
As is well known, the assumption that $S$ be locally closed is necessary in Proposition~\ref{prop:weak-nullstellensatz}. Consider, after~\cite[Rem.\,2.6.5]{BCR}, the set
\[
S=\{(x,y)\in\R^2:y>0\}\cup\{(0,0)\}\,,
\]
and functions $f,g\in\Aa(S)$ defined as $f(x,y)=y$, $g(x,y)=x^2+y^2$. We have $\ZZ_S(f)\subset\ZZ_S(g)$, but $g^k/f$ has no continuous, and thus no arc-analytic, extension to $S$ for any $k\in\N$.
\end{remark}

The following result is a consequence of Lemma~\ref{lem:general-finite-intersection}. (The lemma below was actually claimed in \cite[Lem.\,5.3]{S}, alas with an invalid proof.)

\begin{lemma}
\label{lem:zero-set-by-one-function}
Let $S$ be a semialgebraic set, and let $I$ be an ideal in $\Aa(S)$. There exists a function $f\in I$ such that $\ZZ_S(I)=\ZZ_S(f)$. In particular, $\ZZ_S(I)$ is nonempty if $I$ is a proper ideal.
\end{lemma}

\begin{proof}
By definition, we have
\[
\ZZ_S(I)=\bigcap_{g\in I}\ZZ_S(g)\,.
\]
Since, by Proposition~\ref{prop:arcan-AR-continuous}(ii), each of the sets $\ZZ_S(g)$ is $S$-$\AR$-closed, Lemma~\ref{lem:general-finite-intersection} implies that
\[
\ZZ_S(I)=\ZZ_S(g_1)\cap\dots\cap\ZZ_S(g_t)\,,
\]
for some $g_1,\dots,g_t\in I$. The function $f=g_1^2+\dots+g_t^2$ then has the required properties.
\end{proof}

Below, we denote by $\rad_X(I)$ the radical in $\Aa(X)$ of an ideal $I\subset\Aa(X)$.

\begin{theorem}[\textbf{Arc-Analytic Nullstellensatz}]
\label{thm:general-Nullstellensatz}
Let $S\subset\R^n$ be a locally closed semialgebraic set with the zero-set property, and let $X\subset S$ be $S$-$\AR$-closed.
\begin{itemize}
\item[(i)] If $Y\subset X$ is $X$-$\AR$-closed, then $\ZZ_X(\II_X(Y))=Y$. 
\item[(ii)] If $I$ is an ideal in $\Aa(X)$, then $\II_X(\ZZ_X(I))=\rad_X(I)$.
\end{itemize}
\end{theorem}

\begin{proof} (i) Since $S$ has the zero-set property, given an $X$-$\AR$-closed (and hence $S$-$\AR$-closed) set $Y$, there exists $f\in\Aa(S)$ such that $Y=\ZZ_S(f)$. Then, the restriction $f|_X$ is in $\II_X(Y)$, and so $\ZZ_X(\II_X(Y))\subset\ZZ_X(f|_X)=Y$. The inclusion $\ZZ_X(\II_X(Y))\supset Y$ is obvious.

(ii) Given an ideal $I$ in $\Aa(X)$, we can choose, by Lemma~\ref{lem:zero-set-by-one-function}, a single function $f\in I$ such that $\ZZ_X(I)=\ZZ_X(f)$.
Let $g\in\II_X(\ZZ_X(I))$ be arbitrary. Then, $\ZZ_X(f)\subset\ZZ_X(g)$.
As a closed subset of $S$ (Lemma~\ref{lem:AR-closed-is-closed}), $X$ itself is a locally closed semialgebraic set. It thus follows from Proposition~\ref{prop:weak-nullstellensatz} that $g\in\rad_X((f))$, and hence $g\in\rad_X(I)$. This proves that $\II_X(\ZZ_X(I))\subset\rad_X(I)$. The opposite inclusion follows from the fact that $\II_X(\ZZ_X(I))$ is a radical ideal which contains $I$.
\end{proof}

\begin{corollary}
\label{cor:zero-set-of-prime-is-irred}
Let $S\subset\R^n$ be a locally closed semialgebraic set with the zero-set property, and let $\pp$ be a prime ideal in $\Aa(S)$. Then, the zero-set $\ZZ_S(\pp)$ is $S$-$\AR$-irreducible.
\end{corollary}

\begin{proof}
By Theorem~\ref{thm:general-Nullstellensatz}(ii), $\pp=\II_S(\ZZ_S(\pp))$. Suppose $X_1,\dots, X_r$ are the $S$-$\AR$-irreducible components of $\ZZ_S(\pp)$. Since $S$ has the zero-set property, there are $f_1,\dots,f_r\in\Aa(S)$ such that $X_j=\ZZ_S(f_j)$ for all $j$. Then, $f_1\cdots f_r|_{\ZZ_S(\pp)}=0$, and hence $f_1\cdots f_r\in\II_S(\ZZ_S(\pp))=\pp$. Since $\pp$ is prime, $f_j\in\pp$ for some $j$, and consequently $\ZZ_S(\pp)\subset\ZZ_S(f_j)=X_j$. Therefore, $r=1$ and $\ZZ_S(\pp)=X_1$ is $S$-$\AR$-irreducible.
\end{proof}

\begin{corollary}
\label{cor:irred-iff-prime}
Let $S\subset\R^n$ be a locally closed semialgebraic set with the zero-set property, and let $X\subset S$ be an $S$-$\AR$-closed set. The following properties are equivalent:
\begin{itemize}
\item[(i)] $X$ is $S$-$\AR$-irreducible
\item[(ii)] $X$ is an irreducible semialgebraic set
\item[(iii)] $\II_S(X)$ is prime.
\end{itemize}
\end{corollary}

\begin{proof}
The implication (i)\,$\Rightarrow$\,(ii) follows from Proposition~\ref{prop:s-AR-irred-implies-arcan-irred}, while (ii)\,$\Rightarrow$\,(iii) follows from Proposition~\ref{prop:prime-when-irred}. (iii)\,$\Rightarrow$\,(i), in turn, follows from Corollary~\ref{cor:zero-set-of-prime-is-irred}, since $X=\ZZ_S(\II_S(X))$ in light of Theorem~\ref{thm:general-Nullstellensatz}(i).
\end{proof}

Let us now turn briefly to the maximal spectrum of $\Aa(S)$.

\begin{lemma}[{\cite[Lem.\,5.5]{S}}]
\label{lem:max-Nullstellensatz}
Let $S\subset\R^n$ be a semialgebraic set.
For every maximal ideal $\mi$ in $\Aa(S)$, we have $\II_S(\ZZ_S(\mi))=\mi$.
\end{lemma}

\begin{proof}
That $\mi\subset\II_S(\ZZ_S(\mi))$ \, is clear.
It then suffices by maximality of $\mi$ to observe that $\II_S(\ZZ_S(\mi))$ is a proper ideal.
Indeed, by Lemma~\ref{lem:zero-set-by-one-function}, $\ZZ_S(\mi)$ is not empty, and hence $1\notin\II_S(\ZZ_S(\mi))$.
\end{proof}

Consider the canonical coordinate system $x=(x_1,\dots,x_n)$ of $\R^n$.
For a semialgebraic set $S\subset\R^n$ and a point $a=(a_1,\dots,a_n)\in S$, let $\mi_a=\II_S(\{a\})$ denote the ideal in $\Aa(S)$ of all functions vanishing at $a$. 
In the case that $S$ is locally closed in $\R^n$, Theorem~\ref{thm:general-Nullstellensatz}(ii) gives the description $\mi_a=\rad_S((x_1-a_1,\dots,x_n-a_n))$.
In general, without the local closedness assumption, we have the following.

\begin{proposition}[{\cite[Prop.\,5.6]{S}}]
Let $S$ be a semialgebraic subset of $\R^n$.
The family of maximal ideals in $\Aa(S)$ is given by $\{\mi_a : a\in S\}$.
\end{proposition}

\begin{proof}
First, we show that each $\mi_a=\II_S(\{a\})$ is maximal, where $a=(a_1,\dots,a_n)\in S$.
Let $\mi$ be a maximal ideal in $\Aa(S)$ containing $\mi_a$. Then, $\ZZ_S(\mi)\subset\ZZ_S(\mi_a)$.
Now, $\ZZ_S(\mi_a)\supset\{a\}$, together with the fact that
\[
(x_1-a_1)^2+\cdots+(x_n-a_n)^2\in\mi_a\,,
\]
imply that $\ZZ_S(\mi_a)=\{a\}$. Therefore, since $\ZZ_S(\mi)$ is not empty by Lemma~\ref{lem:zero-set-by-one-function}, we get $\ZZ_S(\mi)=\{a\}$.
It follows by Lemma~\ref{lem:max-Nullstellensatz} that $\mi=\II_S(\{a\})=\mi_a$.

Next, we show that every maximal ideal is included in this family. Let $\mi$ be a maximal ideal in $\Aa(S)$. Because $\ZZ_S(\mi)$ is nonempty, we can choose $a\in\ZZ_S(\mi)$. Then, using Lemma~\ref{lem:max-Nullstellensatz}, we can write
\[
\mi=\II_S(\ZZ_S(\mi))\subset\II_S(\{a\})=\mi_a\,,
\]
whence $\mi=\mi_a$ by maximality.
\end{proof}

To sum up the above discussion, we showed that:

\begin{corollary}
\label{cor-alg-geom-dictionary}
Let $S\subset\R^n$ be a locally closed semialgebraic set with the zero-set property.
The assignment $\{X\mapsto\II_S(X)\}$ establishes a one-to-one correspondence between:
\begin{itemize}
\item[(i)] $S$-$\AR$-closed subsets of $S$ and the radical ideals in $\Aa(S)$
\item[(ii)] $S$-$\AR$-irreducible subsets of $S$ and the prime ideals in $\Aa(S)$
\item[(iii)] Points of $S$ and the maximal ideals in $\Aa(S)$.
\end{itemize}
\end{corollary}

For the next result, recall the notion of Krull dimension in a Noetherian topological space:
Let $X$ be a Noetherian topological space, and let $F\subset X$ be a nonempty irreducible closed set. The \emph{Krull dimension} of $F$, denoted $\dim_K\!F$, is defined as the supremum of lengths $l$ of the chains
\[
F=F_0\varsupsetneq F_1\varsupsetneq F_2\varsupsetneq\dots\varsupsetneq F_l
\]
of irreducible closed sets in $X$. For an arbitrary nonempty closed set $F$, its Krull dimension is the maximum of Krull dimensions of its irreducible components.
By convention, $\dim_K\!\varnothing=-\infty$.

\begin{theorem}
\label{thm:dim-equals-Krull-dim}
Let $S\subset\R^n$ be a locally closed semialgebraic set with the zero-set property. Then,
\[
\dim{S}=\dim_K\!S=\dim\Aa(S)=\dim\Aa(\Cl^\AR_M(S))=\dim_K(\Cl^\AR_M(S))=\dim\Cl^\AR_M(S)\,,
\]
where $\dim_K\!S$ is the topological Krull dimension of $S$ equipped with its intrinsic $S$-$\AR$-topology, $\dim\Aa(S)$ is the Krull dimension of the ring $\Aa(S)$, and $M\subset\R^n$ is an arbitrary Nash manifold containing $S$.
\end{theorem}

\begin{proof}
The Euclidean dimension of $S$ is equal to that of its Zariski closure in $\R^n$. The equalities
\[
\dim{S}=\dim\Aa(\Cl^\AR_M(S))=\dim_K(\Cl^\AR_M(S))=\dim\Cl^\AR_M(S)
\]
thus follow from Kurdyka's~\cite[Prop.\,2.11]{Kur}, since by~\cite[Thm.\,VI.2.1, Rem.\,VI.2.11]{Shiota}, we may assume that the Nash manifold $M$ is an $\AR$-closed subset of some $\R^N$.

Suppose that $k=\dim{S}>0$. Let $S_0$ be an $S$-$\AR$-irreducible component of $S$, of dimension $k$, and let $p\in\Reg_k{S_0}$ be arbitrary. By semialgebraic stratification, for a generic affine hypersurface $H$ through $p$ in $\R^n$, the $S$-$\AR$-closed set $S_0\cap H$ is of dimension $k-1$, and hence it contains an $S$-$\AR$-irreducible set of the same dimension. It follows by induction that $\dim_K\!S\geq k$.

The inequality $\dim_K\!S\leq k$ follows from Corollary~\ref{cor:irred-implies-dim-drop}. 
Lastly, $\dim_K\!S=\dim\Aa(S)$, by Corollary~\ref{cor-alg-geom-dictionary}(ii).
\end{proof}

\vspace{2ex}
%%%%%%%%%%%%%%%%%%%%%%%%%%%%%%%%%%%%%%%%%%%%%%%%%%
%%%%%%%%%%%%%%%%% Section %%%%%%%%%%%%%%%%%%%%%%%%
%%%%%%%%%%%%%%%%%%%%%%%%%%%%%%%%%%%%%%%%%%%%%%%%%%

\section{Generalized arc-analytic extension property}
\label{sec:AEP}

In this section, we study the arc-analytic extension property (AEP, for short; Definition~\ref{def:zero-set-arcan-ext}) and its algebro-geometric consequences for a semialgebraic set. We begin by proving a generalization of Theorem~\ref{thm:ASey-arc-an-extension}, showing that the AEP holds on basic closed semialgebraic sets, and even more generally on all sets defined on a Nash manifold by weak inequalities on arc-analytic functions.

\begin{theorem}
\label{thm:basic-closed-has-AEP}
Let $M$ be a Nash manifold, let $h_1,\dots,h_s\in\Aa(M)$, and let $S\subset M$ be given as
\[
S=\{x\in M:h_1(x)\geq0\;\wedge\dots\wedge\;h_s(x)\geq0\}\,.
\]
Then, $S$ has the arc-analytic extension property.
\end{theorem}

\begin{proof}
We proceed by induction on $s$, the number of defining inequalities of the set $S$. If $s=0$, so that $S=M$, then $S$ has the AEP by Theorem~\ref{thm:ASey-arc-an-extension}. Suppose then that $s\geq1$ and the claim holds for any semialgebraic set of the form $\{x\in N:g_1(x)\geq0\;\wedge\dots\wedge\;g_{s-1}(x)\geq0\}$, where $N$ is a Nash manifold and $g_1,\dots,g_{s-1}\in\Aa(N)$.

Given $S=\{x\in M:h_1(x)\geq0\;\wedge\dots\wedge\;h_s(x)\geq0\}$, define
\[
\wt{S}=\{(x,w)\in M\times\R:h_1(x)\geq0\;\wedge\dots\wedge\;h_{s-1}(x)\geq0\;\wedge\;h_s(x)=w^2\}\,.
\]
By the inductive hypothesis, the set $T=\{(x,w)\in M\times\R:h_1(x)\geq0\;\wedge\dots\wedge\;h_{s-1}(x)\geq0\}$ \;has the arc-analytic extension property. Since $\wt{S}$ is arc-symmetric in $T$ (as the zero-set of an arc-analytic function), it too has the extension property.

Let $\pi:\wt{S}\to S$ be the restriction of the canonical projection $M\times\R\to M$. Let $E\subset S$ be an arbitrary $S$-$\AR$-closed set, and let $f\in\Aa(E)$. Suppose first that $f$ is bounded below, and choose $m\in\R$ such that $f(x)+m\geq1$ for all $x\in E$. Then, the function $g:E\to\R$ defined as
\[
g(x)=\sqrt{\frac{f(x)+m}{2}}
\]
is arc-analytic on $E$. Let $\wt{E}=\pi^{-1}(E)$, and let $\wt{g}:\wt{E}\to\R$ be defined as $\wt{g}=g\circ(\pi|_{\wt{E}})$. Then, $\wt{g}\in\Aa(\wt{E})$, and since $\wt{E}$ is arc-symmetric in $\wt{S}$, there exists by assumption a function $\wt{G}$ arc-analytic on $\wt{S}$ and such that $\wt{G}|_{\wt{E}}=\wt{g}$.
Let $H^+$ and $H^-$ denote the two closed ``hemispheres'' of $\wt{S}$,
\[
H^+=\{(x,w)\in\wt{S}:w\geq0\}\,,\qquad H^-=\{(x,w)\in\wt{S}:w\leq0\}\,,
\]
and let a function $\wt{F}:S\to\R$ be defined as
\begin{equation}
\label{eq:collapsed-function}
\wt{F}(x)\;=\;\left(\wt{G}\circ(\pi|_{H^+})^{-1}(x)\right)^2+\left(\wt{G}\circ(\pi|_{H^-})^{-1}(x)\right)^2\,.
\end{equation}
We claim that $\wt{F}$ is arc-analytic on $S$. To see this, let $\gamma:(-1,1)\to S$ be an arbitrary (non-constant) analytic arc. If $\Int(\gamma^{-1}(\ZZ_M(h_s)))\neq\varnothing$, then $\gamma((-1,1))\subset\ZZ_M(h_s)$, by arc-symmetry of the set $S\cap\ZZ_M(h_s)$ in $S$. The arc $\gamma$ then lifts to $\wt{S}$ as $\wt\gamma(t)=(\gamma(t),0)$, and
\[
(\pi|_{H^+})^{-1}(\gamma(t))=(\pi|_{H^-})^{-1}(\gamma(t))=\wt\gamma(t)\,,\quad\mathrm{for\ all\ }t\in(-1,1)\,.
\]
Consequently, by~\eqref{eq:collapsed-function}, $\wt{F}\circ\gamma=2(\wt{G}\circ\wt\gamma)^2$, and the latter is analytic.
If, in turn, $\gamma$ intersects the zero-set $\ZZ_M(h_s)$ only at isolated points, then the problem reduces to one of the following two cases:\\
Case 1: $\gamma(t)\in S\setminus\ZZ_M(h_s)$ for all $t$.\\
Case 2: $h_s(\gamma(0))=0$, and $h_s(\gamma(t))>0$ for all $t\neq0$. 

In the first case, $\gamma$ lifts to two (disjoint) analytic arcs, $\wt\gamma_1:(-1,1)\to H^+$ and $\wt\gamma_2:(-1,1)\to H^-$, satisfying $\pi|_{H^+}\circ\wt\gamma_1=\gamma$ and $\pi|_{H^-}\circ\wt\gamma_2=\gamma$. By~\eqref{eq:collapsed-function} we then have
\[
\wt{F}\circ\gamma=\left(\wt{G}\circ(\pi|_{H^+})^{-1}\circ(\pi|_{H^+})\circ\wt\gamma_1\right)^2+\left(\wt{G}\circ(\pi|_{H^-})^{-1}\circ(\pi|_{H^-})\circ\wt\gamma_2\right)^2=(\wt{G}\circ\wt\gamma_1)^2+(\wt{G}\circ\wt\gamma_2)^2\,,
\]
and the latter is analytic.
Finally, in Case 2, it suffices to show that $\wt{F}\circ\gamma$ has an analytic germ at the origin. To this end, consider the Taylor expansion at $t=0$ of the analytic function $h_s\circ\gamma(t)$. By virtue of $\gamma$ being non-constant and contained in the locus $\{h_s>0\}$ except at $t=0$, we can write
\[
h_s\circ\gamma(t)\;=\;t^\alpha\cdot u(t)\,,
\]
where $\alpha$ is a positive even integer, and $u(t)$ is analytic and positive for all $t\in(-\ve,\ve)$. The equation $h_s\circ\gamma(t)=w^2(t)$ thus admits two distinct analytic solutions, namely
\[
w_1(t)=t^{\alpha/2}\cdot\sqrt{u(t)}\qquad\mathrm{and}\qquad w_2(t)=-t^{\alpha/2}\cdot\sqrt{u(t)}\,,
\]
which give rise to analytic arcs $\wt\gamma_1,\wt\gamma_2:(-\ve,\ve)\to\wt{S}$ defined as
\[
\wt\gamma_1(t)=(\gamma(t),w_1(t))\qquad\mathrm{and}\qquad\wt\gamma_2(t)=(\gamma(t),w_2(t))\,.
\]
If $\alpha/2$ is even, then $(\pi|_{H^+})^{-1}(\gamma(t))=\wt\gamma_1(t)$ and $(\pi|_{H^-})^{-1}(\gamma(t))=\wt\gamma_2(t)$, for all $t$, and so $\wt{F}\circ\gamma$ is analytic by the arc-analyticity of $\wt{G}$.
If instead $\alpha/2$ is odd, then for $t\leq0$,
\[
\wt{F}\circ\gamma(t)=\left(\wt{G}\circ(\pi|_{H^+})^{-1}(\gamma(t))\right)^2+\left(\wt{G}\circ(\pi|_{H^-})^{-1}(\gamma(t))\right)^2=
(\wt{G}\circ\wt\gamma_2(t))^2+(\wt{G}\circ\wt\gamma_1(t))^2\,,
\]
and for $t\geq0$,
\[
\wt{F}\circ\gamma(t)=\left(\wt{G}\circ(\pi|_{H^+})^{-1}(\gamma(t))\right)^2+\left(\wt{G}\circ(\pi|_{H^-})^{-1}(\gamma(t))\right)^2=
(\wt{G}\circ\wt\gamma_1(t))^2+(\wt{G}\circ\wt\gamma_2(t))^2\,.
\]
In either case, $\wt{F}\circ\gamma$ is equal to $(\wt{G}\circ\wt\gamma_1)^2+(\wt{G}\circ\wt\gamma_2)^2$, which is analytic.
\smallskip

Observe next that, for all $x\in E$, we have
\begin{multline}
\notag
\wt{F}(x)\;=\;\left(\wt{G}((\pi|_{H^+})^{-1}(x))\right)^2+\left(\wt{G}((\pi|_{H^-})^{-1}(x))\right)^2\\
=\;\left(\wt{g}((\pi|_{H^+})^{-1}(x))\right)^2+\left(\wt{g}((\pi|_{H^-})^{-1}(x))\right)^2\\
=\;(g(x))^2+(g(x))^2\;=\;2\cdot\left(\sqrt{\frac{f(x)+m}{2}}\right)^2\;=\;f(x)+m\,.
\end{multline}
Define $F:S\to\R$ as $F=\wt{F}-m$. Then, $F\in\Aa(S)$ and $F|_E=f$, as required.
\smallskip

To complete the proof, suppose that $f\in\Aa(E)$ is not bounded below. By~\cite[Thm.\,VI.2.1, Rem.\,VI.2.11]{Shiota}, we may assume that the Nash manifold $M$ is a closed subset of some $\R^n$. Then, $E$ is closed in $\R^n$, as a closed subset of $S$, which itself is closed in $M$. By polynomial boundedness of continuous semialgebraic functions (see, e.g., \cite[Prop.\,2.6.2]{BCR}), there is a polynomial $P\in\R[x_1,\dots,x_n]$ such that $P(x)\geq1$ for all $x\in\R^n$ and $f/P$ is bounded below on $E$. One can now repeat the above proof to find an arc-analytic extension $F\in\Aa(S)$ of $f/P$. Then, $P\!\cdot\!F$ is the required extension of $f$.
\end{proof}

\begin{remark}
\label{rem:AR-closed-without-description}
As a curiosity, let us mention an example of an arc-symmetric semialgebraic set, for which we don't know an explicit defining arc-analytic function.
Namely, let $S_2$ be the set from Example~\ref{ex: Whitney-in-a-ball-irred-decomp}, which is the two-dimensional irreducible component of the set $S$ in that example. More precisely, we have
\[
S_2=\{(x,y,z)\in B:x^2=zy^2,z\geq0\}\,,
\]
where $B=\{(x,y,z)\in\R^3:\, x^2+(y-1)^2+(z+1)^2\leq2\}$.
If $\wt{B}$ denotes the ``suspension'' $\{(x,y,z,w)\in\R^4:\ 2-x^2-(y-1)^2-(z+1)^2=w^2\}$ of $B$, and $\pi:\wt{B}\to B$ is as in the proof of Theorem~\ref{thm:basic-closed-has-AEP}, then $\pi^{-1}(S_2)$ is an $\AR$-irreducible subset of $\R^4$. We challenge the reader to find an explicit formula for a function $f\in\Aa(\R^4)$ satisfying $\ZZ_{\R^4}(f)=\pi^{-1}(S_2)$.
\end{remark}
\smallskip

Next, we show that a semialgebraic set $S$ all of whose irreducible components (or, all of the $S$-$\AR$-irreducible components) have the AEP, enjoys the AEP as well. More generally, we have the following.

\begin{proposition}
\label{prop:comps-AEP-implies-AEP}
Let $S\subset\R^n$ be a semialgebraic set. Suppose $S=S_1\cup\dots\cup S_r$, where each $S_i$ is an $S$-$\AR$-closed set satisfying the arc-analytic extension property. Then, $S$ has the arc-analytic extension property.
\end{proposition}

\begin{proof}
We proceed by induction on $r$. If $r=1$, there is nothing to show.
Suppose then that $r>1$, $S=S_1\cup\dots\cup S_r$, and the statement holds for all semialgebraic sets $T$ that can be written as a union of at most $r-1$ $T$-$\AR$-closed subsets with the arc-analytic extension property. Define $S'=S_1\cup\dots\cup S_{r-1}$. By the inductive hypothesis, $S'$ has the arc-analytic extension property, since all the $S_1,\dots, S_{r-1}$ are $S'$-$\AR$-closed.

Let $E\subset S$ be an arbitrary $S$-$\AR$-closed set and let $f\in\Aa(E)$. Then $E\cap S'$ is arc-symmetric in $S'$, and $f|_{E\cap S'}\in\Aa(E\cap S')$ as a restriction of an arc-analytic function. By assumption, there exists $F'\in\Aa(S')$ such that $F'|_{E\cap S'}=f|_{E\cap S'}$.

Now, $S'\cap S_n$ is arc-symmetric in $S_n$. Consider the restriction $F'|_{S'\cap S_n}\in\Aa(S'\cap S_n)$. We define a function $G:(E\cup S') \cap S_n\to\R$ as $G=F'|_{S'\cap S_n}\cup f|_{E\cap S_n}$. The function is well-defined as $F'$ is an extension of $f$ on $S'$ and thus, the two agree on the overlap of their domains. We claim that $G\in\Aa((E\cup S') \cap S_n)$. Indeed, this follows from the fact that an analytic arc into $(E\cup S') \cap S_n$ is either entirely contained in $S'\cap S_n$ or in $E\cap S_n$, by the arc-symmetry of these sets.

As $(E\cup S')\cap S_n$ is arc-symmetric in $S_n$, by the arc-analytic extension property of $S_n$, there exists an extension of $G$ to a function $F_n\in\Aa(S_n)$. Then $F:S\to\R$ defined as $F=F'\cup F_n$ is arc-analytic on the whole $S$, since it is arc-analytic on $S'$ and $S_n$ separately. By construction, $F|_E=f$, as required.
\end{proof}

\begin{conjecture}
\label{conj:AEP-for-all}
We conjecture that, more generally, every finite union of sets having the arc-analytic extension property, itself has this property. Consequently, we believe the AEP holds for all locally closed semialgebraic sets (cf.~\cite[Thm.\,2.7.2]{BCR}).
\end{conjecture}
\smallskip

Another class of semialgebraic sets that have the AEP are one-dimensional sets.

\begin{corollary}
\label{cor:1-dim-has-AEP}
Let $S\subset\R^n$ be a semialgebraic set and suppose $\dim{S}=1$. Then, $S$ has the arc-analytic extension property.
\end{corollary}

\begin{proof}
By Proposition~\ref{prop:comps-AEP-implies-AEP}, it suffices to show that every $S$-$\AR$-irreducible component of $S$ has the arc-analytic extension property.
Let $X$ be an $S$-$\AR$-irreducible component of $S$, let $Z\subset X$ be $X$-$\AR$-closed, and let $f\in\Aa(Z)$.
By Proposition~\ref{prop:s-AR-irred-implies-arcan-irred}, $X$ is an irreducible semialgebraic set, and so Corollary~\ref{cor:irreducible-dim-drop} implies that $Z$ is either a finite set, or else $Z=X$.
If $Z$ is finite, a polynomial will extend $f$ to $\Aa(X)$. If not, $f$ is in $\Aa(X)$ already.
\end{proof}

Semialgebraic sets with the arc-analytic extension property also satisfy an arc-analytic variant of the Urysohn lemma.
More precisely, we have the following.

\begin{proposition}[\textbf{Arc-Analytic Urysohn Lemma}]
\label{prop:arcan-Urysohn}
Let $S\subset\R^n$ be a semialgebraic set with the arc-analytic extension property and let $X$ and $Y$ be disjoint $S$-$\AR$-closed sets.
There exists $f\in\Aa(S)$ such that $f|_X=0$ and $f|_Y=1$.
In particular, there are disjoint open semialgebraic sets $U$ and $V$ such that $X\subset U$ and $Y\subset V$.
\end{proposition}

\begin{proof}
The function $f:X\cup Y\to\R$ given by
\[
f(x)=\begin{cases} 0 &\textrm{if\ }x\in X\\ 1 &\text{if\ } x\in Y \end{cases}
\]
is arc-analytic semialgebraic, and $X\cup Y$ is $S$-$\AR$-closed.
By arc-analytic extension, we find a function $F \in\Aa(S)$ such that $F|_{X\cup Y}=f$.
By continuity of $F$, the sets $U=F^{-1}((-\infty,1/2))$ and $V=F^{-1}((1/2,\infty))$ are open semialgebraic.
\end{proof}

We conclude by showing that the arc-analytic extension property always implies the zero-set property. In the classical, arc-symmetric setting, this result can be used as an alternative immediate proof of Theorem~\ref{thm:ASey-zero-set-thm} as a corollary to Theorem~\ref{thm:ASey-arc-an-extension}.

\begin{proposition}
\label{prop:AEP-implies-ZSP}
Let $S\subset\R^n$ be a semialgebraic set satisfying the arc-analytic extension property. Then, $S$ has the zero-set property.
That is, for any $S$-$\AR$-closed set $X\subset S$, there is a function $f\in\Aa(S)$ such that $\ZZ_S(f)=X$.
\end{proposition}

\begin{proof}
Without loss of generality, we may assume that $X$ is $S$-$\AR$-irreducible, and thus an irreducible semialgebraic set (Proposition~\ref{prop:s-AR-irred-implies-arcan-irred}).
We claim that $\ZZ_S(\II_S(X))$ is an irreducible semialgebraic set. Indeed, by irreducibility of $X$ and Proposition~\ref{prop:prime-when-irred}, we have that $\II_S(\ZZ_S(\II_S(X)))=\II_S(X)$ is prime.
Hence, by the AEP assumption and Proposition~\ref{prop:prime-when-irred} again, $\ZZ_S(\II_S(X))$ is irreducible.
Since $\ZZ_S(\II_S(X))\subset\Cl^\Zar(X)$, it follows from Proposition~\ref{prop:alg-closure-dim} that $\dim X=\dim\ZZ_S(\II_S(X))$. Hence, $X=\ZZ_S(\II_S(X))$, by Lemma~\ref{lem:irred-implies-dim-drop} below.
Finally, by Lemma~\ref{lem:zero-set-by-one-function}, there is a function $f\in\II_S(X)$ such that $\ZZ_S(f)=\ZZ_S(\II_S(X))$, and so $X=\ZZ_S(f)$.
\end{proof}

\begin{lemma}[cf. Corollary~\ref{cor:irred-implies-dim-drop}]
\label{lem:irred-implies-dim-drop}
Let $S\subset\R^n$ be a semialgebraic set satisfying the arc-analytic extension property. Suppose $X,Y\subset S$ are $S$-$\AR$-closed sets, $Y\subset X$, and $X$ is irreducible as a semialgebraic set. Then, $\dim{Y}=\dim{X}$ implies $Y=X$.
\end{lemma}

\begin{proof}
Suppose that $Y\subsetneq X$, and let $p\in X\setminus Y$. The sets $\{p\}$ and $Y$ are $S$-$\AR$-closed, hence, by Proposition~\ref{prop:arcan-Urysohn}, there exists $g\in\Aa(S)$ such that $g|_Y=0$ and $g(p)=1$. Then $g|_X\in\Aa(X)$ is not identically zero but $g|_Y=0$. By Proposition~\ref{prop:general-ident-princ}, it follows that $\dim{Y}<\dim{X}$.
\end{proof}

\vspace{2ex}
%%%%%%%%%%%%%%%%%%%%%%%%%%%%%%%%%%%%%%%%%%%%%%%%%%
%%%%%%%%%%%%%%%%% Section %%%%%%%%%%%%%%%%%%%%%%%%
%%%%%%%%%%%%%%%%%%%%%%%%%%%%%%%%%%%%%%%%%%%%%%%%%%

%\section{Zero-set property of locally closed semialgebraic sets}
%\label{sec:zero-set}

%[Some partial results? Particularly: Does ZSP hold on images of arc-symmetric sets by blow-ups, or on Nash sheets?]

\vspace{2ex}
%%%%%%%%%%%%%%%%%%%%%%%%%%%%%%%%%%%%%%%%%%%%%%%%%%
%References
%%%%%%%%%%%%%%%%%%%%%%%%%%%%%%%%%%%%%%%%%%%%%%%%%%
\bibliographystyle{amsplain}

\end{document}